\documentclass[11pt]{article}

\usepackage[english]{babel}
\usepackage{amsfonts,amsmath,amsthm,graphicx,a4wide}
\usepackage[colorlinks=true]{hyperref} % pour pdf

\newtheorem{lem}{Lemma}
\newtheorem{thm}{Theorem}
\newcommand{\half}{\text{\scriptsize $\frac{1}{2}$}}
\newcommand{\tr}[1]{\text{Tr}\left(#1\right)}

\newcommand{\MM}{{\mathbb M}}
\newcommand{\CC}{{\mathbb C}}
\newcommand{\KK}{{\mathbb K}}

\newcommand{\DD}{{\mathbb D}}
\newcommand{\II}{{\mathbb I}}
\newcommand{\EE}[1]{{\text{\large$\mathbb E$}}\left(#1\right)}
\newcommand{\trr}[1]{\text{Tr}^{2}\left(#1\right)}
\newcommand{\bra}[1]{\left<#1\right|}
\newcommand{\ket}[1]{\left|#1\right>}

\newcommand{\bket}[1]{\left<#1\right>}
\newcommand{\nmax}{n^{\text{\tiny max}}}

\newcommand{\PP}[1]{{\text{\large$\mathbb P$}}\left(#1\right)}

\newcommand{\RR}{{\mathbb R}}

\newcommand{\aaa}{{\text{\bf a}}}
\newcommand{\nnn}{{\text{\bf N}}}
\newcommand{\rhop}{\rho^{\text{\tiny pred}}}
\newcommand{\rhope}{\rho^{\text{\tiny pred,est}}}
\newcommand{\rhoe}{\rho^{\text{\tiny est}}}
\newcommand{\chie}{\chi^{\text{\tiny est}}}
\newcommand{\XX}{{\mathcal X}}
\newcommand{\Inv}{{\mathcal I}}

\newtheorem{rem}{Remark}
\newtheorem{prop}{Proposition}

\begin{document}

\title{Stabilization of a delayed quantum system:
\\
the photon box case-study\thanks{This work was   partially supported  by the "Agence Nationale de la Recherche" (ANR),
    "Projet Blanc"  CQUID number 06-3-13957 and "Projet Jeunes Chercheurs" EPOQ2 number ANR-09-JCJC-0070.}}

%%%%%
\author{
  Hadis Amini \thanks{Mines ParisTech: {\tt hadis.amini@mines-paristech.fr}}
  \and
  Mazyar Mirrahimi\thanks{INRIA Paris-Rocquencourt:  {\tt mazyar.mirrahimi@inria.fr} }
\and
  Pierre Rouchon\thanks{Mines ParisTech:  {\tt pierre.rouchon@mines-paristech.fr}}
       }
\date{}
\maketitle
\begin{abstract}
We study  a feedback scheme to stabilize an arbitrary photon number state in a microwave cavity. The quantum non-demolition measurement of the cavity state allows a non-deterministic preparation of Fock states. Here, by the mean of a controlled field injection, we desire to make this preparation process deterministic. The system evolves through a discrete-time Markov process and we design the feedback law applying Lyapunov techniques. Also, in our feedback design we take into account an unavoidable pure delay and we compensate it by a stochastic version of a Smith predictor. After illustrating the efficiency of the proposed feedback law through simulations, we provide a rigorous proof of the global stability of the closed-loop system  based on tools from stochastic stability analysis. A brief study of the Lyapunov exponents of the linearized system around the target state gives a strong indication of the robustness of the method.

\end{abstract}
\section{Introduction}

In the aim of achieving a robust processing of quantum information, one of the main tasks is to prepare and to protect various quantum states. Through the last 15 years, the application of quantum feedback paradigms has been investigated by many physicists~\cite{wiseman-94,vitali-95,doherty-jacobs-99,vanhandel:ieee05,mirrahimi-handel:siam07} as a possible solution for this robust preparation.
However, most (if not all) of these efforts have remained at a theoretical level and have not been able to be give rise to successful experiments. This is essentially due to the necessity of simulating, in parallel to the system, a quantum filter~\cite{belavkin-92} providing an estimate of the state of the system based on the historic of quantum jumps induced by the measurement process. Indeed, it is, in general, difficult to perform such simulations in real time. In this paper, we consider a prototype of physical systems, the photon-box, where we actually have the time to perform these computations in real time (see~\cite{dotsenko-et-al:PRA09} for a detailed description of this  cavity quantum electrodynamics system).

Taking into account the measurement-induced quantum projection postulate, the most practical measurement protocols in the aim of feedback control are the quantum non-demolition (QND) measurements~\cite{braginsky-vorontosov:75,thorne-et-al:78,Unruh-78}. These are the measurements which preserve the value of the measured observable. Indeed, by considering a well-designed QND measurement process where the quantum state to be prepared is an eigenstate of the measurement operator, the measurement process, not only, is not an obstacle for the state preparation but can even help by adding some controllability.

In~\cite{deleglise-et-al:nature08,guerlin-et-al:nature07,gleyzes-et-al:nature07} QND  measures are exploited to detect and/or produce highly non-classical states of light trapped  in a super-conducting cavity (see~\cite[chapter 5]{haroche-raimond:book06} for a  description of such QED  systems and~\cite{brune-et-al:PhRevA92} for detailed physical models with QND measures of light using atoms). For such experimental setups, we detail and analyze  here a  feedback scheme that stabilizes the cavity field towards any photon-number  states (Fock states). Such states are  strongly  non-classical since  their photon numbers are   perfectly defined. The  control corresponds to a  coherent light-pulse injected inside the cavity between atom passages. The overall structure of the proposed feedback scheme is inspired by~\cite{geremia:PRL06} using a  quantum adaptation of the observer/controller structure  widely used for classical  systems (see, e.g.,~\cite[chapter 4]{kailath-book}). As the measurement-induced quantum jumps and the controlled field injection happen in a discrete-in-time manner, the observer part of the proposed feedback scheme consists in a discrete-time quantum filter. Indeed, the discreteness of the measurement process provides us a first prototype of quantum systems where we, actually, have enough time to perform the quantum filtering and to compute the measurement-based feedback law to be applied as the controller.

From a mathematical modeling point of view, the quantum filter evolves through a discrete-time Markov chain.  The estimated state is used  in a state-feedback, based on a Lyapunov design. Indeed, by considering a natural candidate for the Lyapunov function, we propose a feedback law which ensures the decrease of its expectation over the Markov process. Therefore, the value of the  considered Lyapunov function over the Markov chain defines a super-martingale. The convergence analysis of the closed-loop system is, therefore, based on some rather classical tools from stochastic stability analysis~\cite{kushner-71}.

One of the particular features of the system considered in this paper corresponds to a non-negligible delay in the feedback process. In fact, in the experimental setup considered through this paper, we have to take into account a delay of $d$ steps between the measurement process and the feedback injection. Indeed, there are, constantly,  $d$ atoms flying between the photon box (the cavity) to be controlled and the atom-detector (typically $d=5$). Therefore, in our feedback design, we do not have access to the measurement results for the $d$ last atoms. Through this paper, we propose an adaptation of the quantum filter, based on a stochastic version of the Smith predictor~\cite{smith-58}, which takes into account this delay by predicting the actual state of the system without having access to the result of  $d$ last detections.

In the next section, we describe briefly the physical system and the associated quantum Monte-Carlo model.
In Section~\ref{sec:open-loop}, we consider the dynamics of the open-loop system. We will prove, through theorem~\ref{thm:OpenLoop} that the QND measurement process, without any additional controlled injection, allows a non-deterministic preparation of the Fock states. Indeed, we will see that the associated Markov chain converges, necessarily, towards a Fock state and that the probability of converging towards a fixed Fock state is given by its population over the initial state. Also, through proposition~\ref{prop:OpenLoopLin}, we will show that the linearized open-loop system around a fixed Fock state admits strictly negative Lyapunov exponents (see Appendix~\ref{append:Lyap-exp} for a definition of the Lyapunov exponent).

In Section~\ref{sec:closed-loop}, we propose a Lyapunov-based feedback design allowing to stabilize globally the delayed closed-loop system  around a desired Fock state. The theorem~\ref{thm:ClosedLoop} proves the almost sure convergence of the trajectories of the closed-loop system towards the target Fock state. Also, through proposition~\ref{prop:ClosedLoopLin}, we will prove that the linearized closed-loop system around the target Fock state admits strictly negative Lyapunov exponents.

Finally in Section~\ref{sec:filter}, we propose a brief discussion on the considered quantum filter and by proving a rather general separation principle (theorem~\ref{thm:separation}), we will show a semi-global robustness with respect to the knowledge of the initial state of the system. Also, through a  brief analysis of the linearized system-observer around the target Fock state and applying the propositions~\ref{prop:OpenLoopLin} and~\ref{prop:ClosedLoopLin}, we show that its largest Lyapunov exponent is also strictly negative (proposition~\ref{prop:filterlin}).

A preliminary version of this paper without delay  has appeared as a conference paper~\cite{mirrahimi-et-al:cdc09}. The delay compensation scheme is borrowed from~\cite{dotsenko-et-al:PRA09}.  The authors thank
M. Brune,  I. Dotsenko, S. Haroche and J.M. Raimond from ENS for many interesting discussions and advices.

\section{A discrete-time  Markov process}\label{sec:model}

\begin{figure}[htb]
 \centerline{\includegraphics[width=0.8\textwidth]{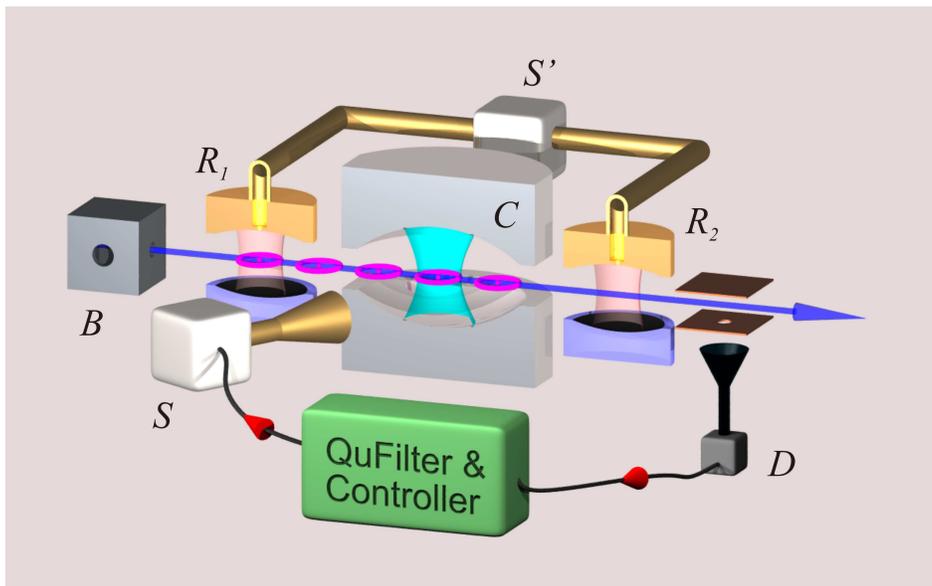}}
  \caption{The  quantum electrodynamic setup including  the microwave cavity $C$ with its  feedback scheme (in green).}\label{fig:ExpScheme}
\end{figure}
As illustrated by Figure~\ref{fig:ExpScheme},
the system consists in $C$ a high-Q microwave cavity, $B$ a box producing Rydberg atoms, $R_1$ and $R_2$ two low-Q Ramsey
cavities, $D$ an atom detector and $S$ a microwave source. The dynamics model is discret in time and relies on quantum Monte-Carlo trajectories (see~\cite[chapter 4]{haroche-raimond:book06}).
Each time-step indexed by the integer $k$ corresponds to  atom number $k$ coming from $B$, submitted then to a first Ramsey $\pi/2$-pulse in $R_1$,  crossing the cavity $C$  and being entangled with it, submitted to a second $\pi/2$-pulse in $R_2$  and finally being measured in $D$. The state of the cavity is associated  to a quantized mode. The control corresponds to a coherent displacement of amplitude $\alpha \in\CC$  that is applied  via the micro-wave source $S$ between two atom passages.

 In this paper we consider a finite dimensional approximation of this quantized mode  and take a truncation to $\nmax$ photons. Thus the cavity  space is approximated by the Hilbert space $\CC^{\nmax+1}$. It admits  $(\ket{0}, \ket{1}, \ldots, \ket{\nmax})$ as ortho-normal basis. Each basis vector $\ket{n}\in \CC^{\nmax+1}$ corresponds to  a pure state, called Fock state, where the cavity has exactly $n$~photons, $n\in\{0,\ldots,\nmax\}$. In this Fock-states basis the number operator $\nnn$ corresponds to the diagonal matrix
 $$
 \nnn=\text{diag}(0,1, \ldots, \nmax).
 $$
 The annihilation operator truncated to $\nmax$ photons is denoted by $\aaa$. It  corresponds  to the upper $1$-diagonal matrix filled with $(\sqrt{1}, \ldots, \sqrt{\nmax})$:
 $$
 \aaa \ket{0} = 0, \quad \aaa \ket{n} = \sqrt{n} \ket{n-1} \text{~for~} n=1,\ldots,\nmax
 $$
 The  truncated creation operator denoted by  $\aaa^\dag$ is the Hermitian conjugate of $\aaa$. Notice that we still have $\nnn=\aaa^\dag\aaa$, but truncation does not preserve the usual commutation $[\aaa,\aaa^\dag]=1$ that is only valid when $\nmax=+\infty$.

Just after the measurement of the atom number $k-1$, the state of the cavity is described  by the density matrix $\rho_{k}$ belonging to  the following set of well-defined density matrices:
\begin{equation}\label{eq:XX}
\XX=\left\{\rho\in{\mathbb C^{(\nmax+1)^2}}\quad| \quad\rho=\rho^\dag,~\tr{\rho}=1,~\rho\geq 0\right\}.
\end{equation}
The random evolution of this   state $\rho_k$ can be modeled through a discrete-time  Markov process that will be described bellow (see~\cite{dotsenko-et-al:PRA09} and the references therein explaining the physical modeling assumptions).

 Let us denote by  $\alpha_k\in\CC$ the control at step $k$. Then $\rho_{k+1}$, the cavity state after measurement of atom $k$ is given by
\begin{equation}\label{eq:main}
\rho_{k+1}=\MM_{s_k}(\rho_{k+\half}), \quad \rho_{k+\half}=\DD_{\alpha_{k-d}}(\rho_k)
\end{equation}
 where,
\begin{itemize}
\item $s_k\in\{g,e\}$, $\MM_{g}(\rho)=\tfrac{M_{g}\rho M_{g}^\dag}{\tr{M_{g}\rho M_{g}^\dag}}$, $\MM_{e}(\rho)=\tfrac{M_{e}\rho M_{e}^\dag}{\tr{M_{e}\rho M_{e}^\dag}}$
 with operators  $M_g=\cos\left(\varphi_0+{\vartheta}\nnn\right)$ and $M_e=\sin\left(\varphi_0+{\vartheta}\nnn\right)$ ($\varphi_0,{\vartheta}$ constant parameters). For any $n\in\{0,\ldots,\nmax\}$ we set  $\varphi_n= \varphi_0+n {\vartheta}$.
 \item $\DD_{\alpha}(\rho)=D_\alpha\rho D_{\alpha}^\dag$ where the unitary displacement operator $D_\alpha$ is
given by $ D_\alpha=e^{\alpha \aaa^\dag - \alpha^*\aaa}.$ In open-loop, $\alpha=0$, $D_0=\II$ (identity operator) and
$\DD_0(\rho)=\rho$. Notice that $D_\alpha^\dag = D_{-\alpha}$.

\item  $s_k$ is  a random variable taking the value $g$ when the atom $k$ is detected in $g$ (resp. $e$ when the atom $k$ is detected in $e$) with probability
\begin{equation}\label{eq:PgPe}
P_{g,k}=\tr{M_g\rho_{k+\half} M_g^\dag}\quad \left( \text{resp.}~P_{e,k}=\tr{M_e\rho_{k+\half} M_e^\dag} \right).
\end{equation}

\item The control elaborated at step $k$,  $\alpha_{k}$,   is subject  to  a delay of $d$ steps, $d$ being the number of flying atoms between the cavity $C$ and the detector $D$.
\end{itemize}
We will assume through out the paper that
the  parameters $\varphi_0$, $\vartheta$ are chosen in order to have  $M_g$, $M_e$ invertible and  such that the  spectrum of $M_g^\dag M_g= M_g^2$ and $M_e^\dag M_e= M_e^2$   are not degenerate.
This implies that the  nonlinear operators $\MM_g$ and $\MM_e$ are well defined for all $\rho\in \XX$ and that $\MM_g(\rho)$ and $\MM_e(\rho)$ belongs also to the state space $\XX$ defined by~\eqref{eq:XX}.
Notice that $M_g$ and $M_e$ commute, are diagonal in the Fock basis and satisfy  $M_g^\dag M_g+M_e^\dag M_e=\II$.
The Kraus map  associated to this Markov process is given by:
\begin{equation}\label{eq:kraus}
\KK_{\alpha}(\rho) = M_g D_\alpha \rho D_{\alpha}^\dag M_g^\dag + M_e D_\alpha \rho D_{\alpha}^\dag M_e^\dag
.
\end{equation}
It corresponds to the  expectation value of $\rho_{k+1}$ knowing $\rho_k$ and $\alpha_{k-d}$:
\begin{equation}\label{eq:expectation}
\EE{\rho_{k+1}~ | ~\rho_k,\alpha_{k-d}}
= \KK_{\alpha_{k-d}}(\rho_k)
.
\end{equation}

\section{Open loop dynamics}\label{sec:open-loop}

\subsection{Simulations}
\begin{figure}
\centerline{\includegraphics[width=0.6\textwidth]{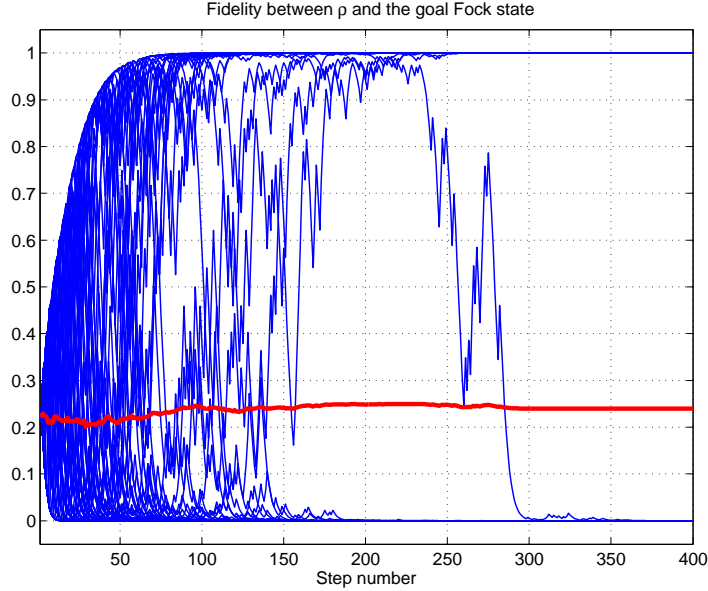}}
  \caption{ $\bket{3|\rho_k|3}$ (fidelity with respect to the $3$-photon state)  versus the number of passing atoms $k\in\{0, \ldots, 400\}$ for 100 realizations of  the open-loop Markov process~\eqref{eq:OpenLoop} (blue fine curves)  starting from the same  coherent state $\rho_0=\DD_{\sqrt{3}}(\ket{0} \bra{0})$. The ensemble average over these realizations corresponds to the thick red curve.}
   \label{fig:OpenLoop}
\end{figure}
We consider in this section the following dynamics
\begin{equation}\label{eq:OpenLoop}
\rho_{k+1}=\MM_{s_k}(\rho_k),
\end{equation}
obtained from~\eqref{eq:main} when $\alpha_{k-d}\equiv 0$.
 Figure~\ref{fig:OpenLoop} corresponds to 100 realizations of this Markov process with
$\nmax = 10$ photons, ${\vartheta}= \tfrac{2}{10}$ and $\varphi_0=\tfrac{\pi}{4}-3{\vartheta} $.  For each realization,  $\rho_0$ is initialized to the the same coherent state $\DD_{\sqrt{3}}(\ket{0}\bra{0})$  with $\tr{\nnn\rho_0}\approx 3$ as mean photon number. We observe that either $\bket{3|\rho_k|3}$ tends to $1$ or $0$. Since the ensemble average curve is almost  constant,  the proportion  of trajectories for which $\bket{3|\rho_k|3}$ tends to $1$  is  given approximatively  by $\bket{3|\rho_0|3}$.

\subsection{Global convergence analysis}
 The following theorem underlies the observations made for  simulations of Figure~\ref{fig:OpenLoop}
\begin{thm}\label{thm:OpenLoop}
Consider the Markov process  $\rho_k$ obeying~\eqref{eq:OpenLoop} with an initial condition $\rho_0\in \XX$ defined by~\eqref{eq:XX}. Then
\begin{itemize}
    \item  for any $n \in\{0,\cdots, \nmax\}$, $\tr{\rho_k \ket{n}\bra{n} }=\bra{n}\rho_k\ket{n}$ is a martingale
    \item $\rho_k$ converges with probability $1$ to one of the $\nmax+1$  Fock state $\ket{n}\bra{n}$ with $n\in\{0,\cdots,\nmax\}$.
    \item the probability to converge  towards the Fock state $\ket{n}\bra{n}$ is given by $\tr{\rho_0 \ket{n}\bra{n}}=\bra{ n}\rho_0\ket{ n}$.
\end{itemize}
\end{thm}

\begin{proof}
Let us prove that  $\tr{\rho_k\ket{n}\bra{n}}$ is a martingale. Set $\xi=\ket{n}\bra{n}$.
We have
\begin{multline*}
\EE{\tr{\xi\rho_{k+1}}~|~\rho_k} =
P_{g,k} \tr{\xi \tfrac{M_g\rho_k M_g^\dag}{P_{g,k}}}
+ P_{e,k} \tr{\xi \tfrac{M_e\rho_k M_e^\dag}{P_{e,k}}}
\\
= \tr{\xi M_g\rho_k M_g^\dag} + \tr{\xi M_e\rho_k M_e^\dag}
= \tr{\rho_k (M^\dag_g\xi M_g+M_e^\dag \xi M_e)}
.
\end{multline*}
Since $\xi$ commutes with $M_g$ and $M_e$
 and $M^\dag_g M_g+M^\dag_e M_e=\II$ we have $\EE{\tr{\xi\rho_{k+1}}~|~\rho_k} = \tr{\xi\rho_k}$.

Considering the following function:
$$
V_n(\rho)=f(\bket{n|\rho|n}),
$$
where $f(x)=\tfrac{x+x^2}{2}.$
Notice that $f$ is $1$-convexe, $f^\prime \geq  \tfrac{1}{2}$ on $[0,1]$ and satisfies
\begin{equation}\label{eq:convexe}
\forall (x,y,\theta)\in[0,1],\quad  \theta f(x) + (1-\theta)f(y) = \tfrac{\theta(1-\theta)}{2}(x-y)^2 + f(\theta x+ (1-\theta) y)
.
\end{equation}
The function $f$ is  increasing and convex and $\bket{n|\rho_k|n}$ is a martingale. Thus $V_n(\rho_k)$ is sub-martingale.
Since
$$
\bket{n|\MM_g(\rho)|n} =
\tfrac{\cos^2\!\varphi_n}{\tr{M_g\rho M_g^\dag}}\bket{n|\rho|n},
\quad \bket{n|\MM_e(\rho)|n} =
\tfrac{\sin^2\!\varphi_n}{\tr{M_e\rho M_e^\dag}}\bket{n|\rho|n} $$
we have
\begin{multline*}
\EE{V_n(\rho_{k+1})~|~\rho_k}=
\tr{M_g\rho_k M_g^\dag}f\left(\tfrac{\cos^2\!\varphi_n}{\tr{M_g\rho_k M_g^\dag}}\bket{n|\rho_k|n}\right)
\\+
\tr{M_e\rho_k M_e^\dag}f\left(\tfrac{\sin^2\!\varphi_n}{\tr{M_e\rho_k M_e^\dag}}\bket{n|\rho_k|n}\right)
\end{multline*}
Then~\eqref{eq:convexe}, together with
$$\theta =\tr{M_g\rho_k M_g^\dag},~x=\tfrac{\cos^2\!\varphi_n}{\tr{M_g\rho_k M_g^\dag}}\bket{n|\rho_k|n}, ~
y=\tfrac{\sin^2\!\varphi_n}{\tr{M_e\rho_k M_e^\dag}}\bket{n|\rho_k|n}
$$
yields to
\begin{multline*}
\EE{V_n(\rho_{k+1})~|~\rho_k}-V_n(\rho_k) = \\
\tfrac{\tr{M_g\rho_k M_g^\dag}\tr{M_e\rho_k M_e^\dag}(\bket{n|\rho_k|n})^2}{2}\left(\tfrac{\cos^2\!\varphi_n}{\tr{M_g\rho_k M_g^\dag}}-\tfrac{\sin^2\!\varphi_n}{\tr{M_e\rho_k M_e^\dag}}\right)^2
.
\end{multline*}
Thus we recover that $V_n(\rho_k)$ is a sub-martingale,  $\EE{V_n(\rho_{k+1})~|~\rho_k} \geq V_n(\rho_k) $.  We  have also shown that
$ \EE{V_n(\rho_{k+1})~|~\rho_k}=V_n(\rho_k)$ implies that either $\bket{n|\rho_k|n}=0$ or
  $\tr{M_g\rho_k M_g^\dag}=\cos^2\!\varphi_n$ (assumption $M_g$ and $M_e$ invertible is used here).

We apply now the invariance theorem established by Kushner~\cite{kushner-71} (recalled in the Appendix~\ref{append:stoch-stab}) for the Markov process $\rho_k$ and the sub-martingale $V_n(\rho_k)$. This theorem implies that the Markov process $\rho_k$ converges in probability to the largest invariant subset of
$$
\left\{\rho\in \XX~|~\tr{M_g\rho M_g^{\dag}}=\cos^2\!\varphi_n \text{ or } \bket{n|\rho|n}=0\right\}.
 $$
 But the set  $\left\{\rho\in \XX~|~\bket{n|\rho|n}=0\right\}$ is invariant. It remains thus  to characterized the largest invariant subset denoted by $\XX_n$  and included in $\left\{\rho\in \XX~|~\tr{M_g\rho M_g^{\dag}}=\cos^2\!\varphi_n \right\}$.

Take $\rho\in \XX_n$. Invariance means that $\MM_g(\rho)$ and $\MM_e(\rho)$  belong to $\XX_n$ (the fact that $M_g$ and $M_e$ are  invertible   ensures that probabilities  to jump with $s=g$ or $s=e$ are strictly positive for any $\rho\in \XX$).
 Consequently  $\tr{M_g\MM_g(\rho)M_g^\dag}=\tr{M_g\rho M_g^{\dag}}=\cos^2\!\varphi_n $.  This means that
 $\tr{M_g^4\rho}=\trr{M_g^2\rho}$. By Cauchy-Schwartz inequality,
 $$
 \tr{M_g^4\rho}=\tr{M_g^4\rho}\tr{\rho}\geq \trr{M_g^2\rho}
 $$
with equality if, and only if, $M_g^4\rho$ and $\rho$ are co-linear.  $M_g^4$ being non-degenerate, $\rho$ is necessarily a projector over an eigenstate of $M_g^4$, i.e., $\rho=\ket{m}\bra{m}$ for some $m\in\{0,\ldots,\nmax\}$. Since $ \tr{M_g\rho M_g^{\dag}}=\cos^2\!\varphi_n >0$, $m=n$ and thus $\XX_n$ is reduced to $\{\ket{n}\bra{n}\}$. Therefore the only possibilities for the $\omega$-limit set are $\tr{\rho\ket{n}\bra{n}}=0$ or $1$ and
 $$
 W_n(\rho_k)=\tr{\rho_k\ket{n}\bra{n}}(1-\tr{\rho_k\ket{n}\bra{n}}\stackrel{k\rightarrow \infty}{\longrightarrow} 0\qquad \text{in probability.}
 $$
 The convergence in probability together with the fact that $W_n(\rho_k)$ is a positive bounded ($W_n\in[0,1]$) random process implies the convergence in expectation. Indeed
 \begin{align*}
 \limsup_{k\rightarrow\infty} \EE{W_n(\rho_k)} &\leq \epsilon \limsup_{k\rightarrow\infty}\PP{W_n(\rho_k)\leq\epsilon}+\limsup_{k\rightarrow\infty}\PP{W_n(\rho_k)>\epsilon}\\
 &\leq \epsilon+\limsup_{k\rightarrow\infty}\PP{W_n(\rho_k)>\epsilon}\leq \epsilon,
 \end{align*}
 where for the last inequality, we have applied the convergence in probability of $W_n(\rho_k)$ towards $0$. As the above inequality is valid for any $\epsilon>0$, we have
 $$
 \lim_{k\rightarrow\infty} \EE{W_n(\rho_k)}=0.
 $$
 Furthermore, by the first part of the Theorem, we know that $\tr{\rho_k\ket{n}\bra{n}}$ is a bounded martingale and therefore by the Doob's first martingale convergence theorem (see the Theorem~\ref{thm:conv_martingale1} of the Appendix~\ref{append:stoch-stab}), $\tr{\rho_k\ket{n}\bra{n}}$ converges almost surely towards a random variable $l_n^\infty\in[0,1]$. This implies that $W_n(\rho_k)$ converges almost surely towards the random variable $l_n^\infty(1-l_n^\infty)\in[0,1]$. We apply now the dominated convergence theorem
 $$
 \EE{l_n^\infty(1-l_n^\infty)}=\EE{ \lim_{k\rightarrow\infty} W_n(\rho_k)}=\lim_{k\rightarrow\infty} \EE{W_n(\rho_k)}=0.
 $$
 This implies that $l_n^\infty(1-l_n^\infty)$ vanishes  almost surely and therefore
 $$
 W_n(\rho_k)=\tr{\rho_k\ket{n}\bra{n}}(1-\tr{\rho_k\ket{n}\bra{n}})\stackrel{k\rightarrow \infty}{\longrightarrow} 0 \qquad \text{almost surely.}
 $$
As we can repeat this same analysis for any choice of $n\in\{0,1,\ldots,\nmax\}$, $\rho_k$ converges almost surely to the set of of Fock states
$$
\{\ket{n}\bra{n}~|~n=0,1,\ldots,\nmax\},
$$
which ends the proof of the second part.

We have shown that the probability measure associated to the random variable $\rho_k$ converges to the probability measure
$$
\sum_{n=0}^{\nmax} p_n\delta(\ket{n}\bra{n}),
$$
where $ \delta(\ket{n}\bra{n})$ denotes the Dirac distribution at $\ket{n}\bra{n}$ and $p_n$ is the probability of convergence towards $\ket{n}\bra{n}$. In particular, we have
$$
\EE{\tr{\ket{n}\bra{n}\rho_k}}\stackrel{k\rightarrow \infty}{\longrightarrow} p_n.
$$
But $\tr{\ket{n}\bra{n}\rho_k}$ is a martingale and $\EE{\tr{\ket{n}\bra{n}\rho_k}}=\EE{\tr{\ket{n}\bra{n}\rho_0}}$. Thus
$$
p_n=\bra{n}\rho_0\ket{n},
$$
which ends the proof of the third and last part.

%%%%%%%%%%%

\end{proof}

\subsection{Local convergence rate}\label{ssec:open-rate}
According to theorem~\ref{thm:OpenLoop}, the $\Omega$-limit set of the Markov process~\eqref{eq:OpenLoop} is the discrete set of Fock states $\{ \ket{n}\bra{n}\}_{n\in\{0,\ldots,\nmax\}}$. We investigate here the local convergence rate around one of these Fock states  denoted by $\bar\rho = \ket{\bar n}\bra{\bar n}$ for some $\bar n \in\{0,\ldots,\nmax\}$.

Since $\MM_g(\bar\rho)=\MM_e(\bar\rho)=\bar\rho$, we can  develop the dynamics~\eqref{eq:OpenLoop} around the fixed point $\bar\rho$. We write $\rho=\bar\rho+\delta\rho$ with $\delta\rho$ small, Hermitian and with zero trace. Keeping only the  first order terms in~\eqref{eq:OpenLoop}, we have
\begin{equation*}
\delta\rho_{k+1}=\tfrac{M_{s_k}\delta\rho_k M_{s_k}^\dag}{\tr{M_{s_k}\bar\rho M_{s_k}^\dag}}-\tfrac{\tr{M_{s_k}\delta\rho_k M_{s_k}^\dag}}{\tr{M_{s_k}\bar\rho M_{s_k}^\dag}}\bar\rho.
\end{equation*}
Thus the linearized  Markov process  around the fixed point $\bar \rho$   reads
\begin{equation}\label{eq:OpenLoopLin}
\delta\rho_{k+1}=A_{s_k}\delta\rho_k A_{s_k}^\dag-\tr{A_{s_k}\delta\rho_k A_{s_k}^\dag}\bar\rho
\end{equation}
where the random matrices $A_{s_k}$ are given by :
\begin{itemize}
\item $A_g=\frac{M_g}{\cos\varphi_{\bar n}}$ with probability $P_g=\cos^2\!\varphi_{\bar n}$
\item $A_e=\frac{M_e}{\sin\varphi_{\bar n}}$ with probability $P_e=\sin^2\!\varphi_{\bar n}$.
\end{itemize}
The following proposition  shows that the convergence of the linearized dynamics is exponential (a crucial  robustness indication).
\begin{prop}\label{prop:OpenLoopLin}
Consider the linear Markov chain~\eqref{eq:OpenLoopLin} of state $\delta\rho$ belonging to the set of Hermitian matrices with zero trace.  Then the  largest  Lyapunov exponent $\Lambda$ is given by ($\varphi_n=\varphi_0+n {\vartheta}$)
$$
\Lambda = \max_{ \text{\scriptsize $\begin{array}{c}
                   n\in\{0,\ldots,\nmax\} \\
                   n\neq \bar n
                 \end{array}$}
}
 \left(\cos^2\!\varphi_{\bar n}\log\left(\tfrac{|\cos\varphi_{n}|}{|\cos\varphi_{\bar n}|}\right)
    +\sin^2\!\varphi_{\bar n}\log\left(\tfrac{|\sin\varphi_n|}{|\sin\varphi_{\bar n}|}\right) \right)
$$
and is  strictly negative: $\Lambda < 0$.
\end{prop}
\begin{proof} Set $\delta\rho^{n_1,n_2}=\bket{n_1|\delta\rho|n_2}$ for any $n_1, n_2\in\{0, \ldots, \nmax\}$. Since $\tr{\delta\rho_k} \equiv 0$, we  exclude here the case  $(n_1,n_2)=(\bar n, \bar n)$ because
$\delta\rho_k^{\bar n, \bar n} = - \sum_{n\neq \bar n} \delta\rho_k^{n, n}$.
Since $A_e$ and $A_g$ are diagonal matrices, we have
\begin{equation}\label{eq:deltarho}
\delta\rho_{k+1}^{n_1,n_2}= a_{s_k}^{n_1,n_2} \delta\rho_k^{n_1,n_2}
\end{equation}
where $s_k=g$ (resp. $s_k=e$) with probability  $\cos^2\varphi_{\bar n}$ (resp. $\sin^2\varphi_{\bar n}$) and where
$a_{g}^{n_1,n_2} =\tfrac{\cos\varphi_{n_1}\cos\varphi_{n_2}}{\cos^2\varphi_{\bar n}}$ and
$a_{e}^{n_1,n_2} =\tfrac{\sin\varphi_{n_1}\sin\varphi_{n_2}}{\sin^2\varphi_{\bar n}}$.

Denote by $\Lambda^{n_1,n_2}$ the Lyapunov exponent of~\eqref{eq:deltarho} for $(n_1,n_2)\neq(\bar n, \bar n)$.
By the law of large numbers,  we know that
$\frac{\log\left(\left|\prod_{l=0}^{l=k}  a_{s_k}^{n_1,n_2}\right|\right)}{k+1}$ converges almost surely towards
$$
\cos^2\!\varphi_{\bar n}~ \log(|a_{g}^{n_1,n_2} |) +
\sin^2\!\varphi_{\bar n} ~\log(|a_{e}^{n_1,n_2} |) .
$$
Thus, we have
\begin{multline*}
\Lambda^{n_1,n_2}=
  \cos^2\varphi_{\bar n}
   \left(\log\left(\tfrac{|\cos\varphi_{n_1}|}{|\cos\varphi_{\bar n}|}\right)+
    \log\left(\tfrac{|\cos\varphi_{n_2}|}{|\cos\varphi_{\bar n}|}\right) \right)
\\
+
\sin^2\varphi_{\bar n}
     \left(\log\left(\tfrac{|\sin\varphi_{n_1}|}{|\sin\varphi_{\bar n}|}\right)+
    \log\left(\tfrac{|\sin\varphi_{n_2}|}{|\sin\varphi_{\bar n}|}\right) \right)
\end{multline*}
The function
\begin{eqnarray*}
    \left[0,\frac{\pi}{2}\right] \ni \varphi \mapsto
    \left(\frac{\cos\varphi}{|\cos\varphi_{\bar n}|}\right)^{\cos^2\varphi_{\bar n}}
    ~
    \left(\frac{\sin\varphi}{|\sin\varphi_{\bar n}|}\right)^{\sin^2\varphi_{\bar n}}
\end{eqnarray*}
increases strictly from $0$ to $1$ when $\varphi$ goes from $0$ to $\arcsin(|\sin\varphi_{\bar n}|)$
and decreases strictly from $1$ to $0$ when $\varphi$ goes from $\arcsin(|\sin\varphi_{\bar n}|)$ to
$\frac{\pi}{2}$. Since $(n_1,n_2)\neq (\bar n, \bar n )$, $\Lambda^{n_1,n_2} <0$. Denote by $\Lambda^n = \Lambda^{n,\bar n}$ for $n\in \{0, \ldots, \nmax\}$:
$$
\Lambda^n = \cos^2\!\varphi_{\bar n}\log\left(\tfrac{|\cos\varphi_{n}|}{|\cos\varphi_{\bar n}|}\right)
    +\sin^2\!\varphi_{\bar n}\log\left(\tfrac{|\sin\varphi_n|}{|\sin\varphi_{\bar n}|}\right)
.
$$
Since $(n_1,n_2)\neq (\bar n, \bar n)$, we have $\Lambda^{n_1,n_2} \leq \max_{n\neq \bar n}  \Lambda^n$ and $ \Lambda = \max_{n\neq \bar n}  \Lambda^n $ is strictly negative.
\end{proof}

\section{Feedback stabilization with delays}\label{sec:closed-loop}
\subsection{Feedback scheme and closed-loop simulations }

\begin{figure}
\centerline{\includegraphics[width=.6\textwidth]{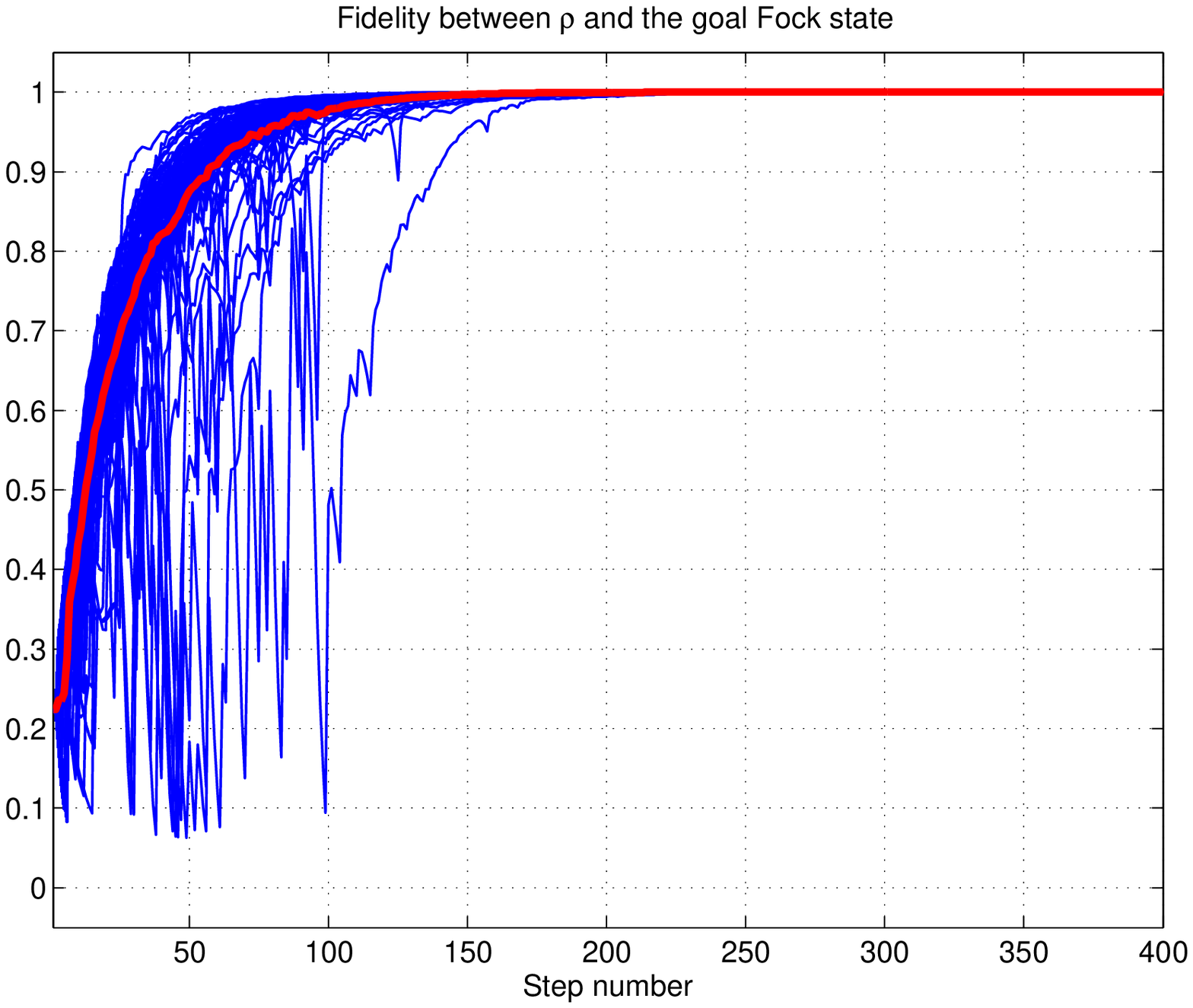}}
  \caption{ $\tr{\rho_k\bar\rho}=\bket{3|\rho_k|3}$  versus $k\in\{0, \ldots, 400\}$ for 100 realizations of  the closed-loop Markov process~\eqref{eq:main} with feedback~\eqref{eq:feedbackCDC} (blue fine curves)  starting from the same  state $\rho_0=\DD_{\sqrt{3}}(\ket{0} \bra{0})$ (no delay, $d=0$). The ensemble average over these realizations corresponds to the thick red curve.}
   \label{fig:ClosedLoop}
\end{figure}

\begin{figure}
\centerline{\includegraphics[width=.6\textwidth]{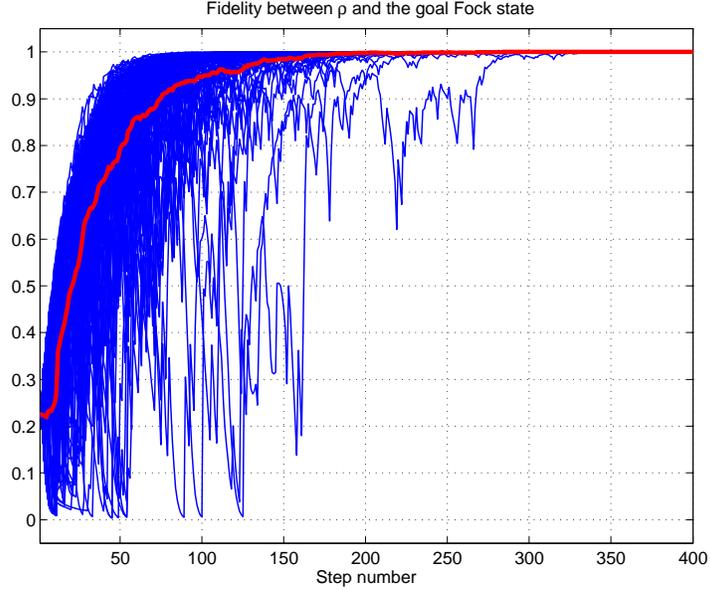}}
  \caption{ $\tr{\rho_k\bar\rho}=\bket{3|\rho_k|3}$  versus $k\in\{0, \ldots, 400\}$ for 100 realizations of  the closed-loop Markov process~\eqref{eq:chi} with feedback~\eqref{eq:feedmain} (blue fine curves)  starting from the same  state $\chi_0=(\DD_{\sqrt{3}}(\ket{0} \bra{0}),0,\ldots,0)$ and with $5$-step delay ($d=5$). The ensemble average over these realizations corresponds to the thick red curve.}
   \label{fig:ClosedLoopDelay}
\end{figure}

Through out this section we assume that we have access at each step $k$ to the cavity  state $\rho_k$. The goal is to design a causal  feedback law  that stabilizes globally the Markov chain~\eqref{eq:main} towards a goal Fock state  $\bar\rho =\ket{\bar n}\bra{\bar n}$  with $\bar n$ photon(s), $\bar n \in \{0,\ldots, \nmax\}$. To be consistent with truncation to $\nmax$ photons, $\bar n$  has to be far from $\nmax$ (typically $\bar n =3$ with $\nmax =10$ in the simulations below).

The feedback is based on the fact that, in open-loop when $\alpha_k\equiv0$, $\tr{\bar\rho\rho_k}=\bket{\bar n | \rho_k | \bar n}$ is a martingale. When $d=0$,~\cite{mirrahimi-et-al:cdc09} proves global almost sure convergence of the following feedback law
\begin{equation}\label{eq:feedbackCDC}
    \alpha_k=\left\{
    \begin{array}{ll}
        \epsilon \tr{\bar\rho~ [\rho_{k},\aaa]}
         & \mbox{ if } \tr{\bar\rho\rho_k} \ge \eta
         \\
     \underset{|\alpha|\le\bar\alpha}{\text{argmax}}
        \left( \tr{\bar\rho~\DD_\alpha(\rho_{k})}\right)
          & \mbox{ if } \tr{\bar\rho\rho_k} < \eta \\
    \end{array}
    \right.
\end{equation}
for any $\bar\alpha >0$  when $\epsilon,\eta >0$  are small enough. This feedback law  ensures that $\tr{\bar\rho\rho_k}$ is a sub-martingale.

When $d>0$, we cannot set
$\alpha_{k-d}= \epsilon \tr{\bar\rho~[\rho_{k},\aaa]}$
since $\alpha_k$ will depend on $\rho_{k+d}$ and the feedback law is not causal. In~\cite{dotsenko-et-al:PRA09}, this feedback law is made causal by replacing $\rho_{k+d}$ by its  expectation value (average prediction) $\rhop_{k}$ knowing $\rho_k$ and the past controls $\alpha_{k-1}$, \ldots, $\alpha_{k_d}$:
$$
  \rhop_{k}= \KK_{\alpha_{k-1}}\circ\ldots \circ \KK_{\alpha_{k-d}}(\rho_k)
$$
where the Kraus map $\KK_\alpha$ is defined by~\eqref{eq:kraus}.

We will thus consider here the following causal feedback based on an average compensation of the delay $d$
\begin{equation}\label{eq:feedmain}
    \alpha_{k}=\left\{
    \begin{array}{ll}
        \epsilon \tr{\bar\rho~ [\rhop_{k},\aaa]}
         & \mbox{ if } \tr{\bar\rho\rhop_{k}} \ge \eta
         \\
     \underset{|\alpha|\le\bar\alpha}{\text{argmax}}
        \left( \tr{\bar\rho~\DD_\alpha(\rhop_{g,k})}\tr{\bar\rho~\DD_\alpha(\rhop_{e,k})}\right)
          & \mbox{ if } \tr{\bar\rho\rhop_{k}} < \eta \\
    \end{array}
    \right.
\end{equation}
with $$
\left\{
\begin{array}{l}
  \rhop_{g,k}=\KK_{\alpha_{k-1}}\circ\ldots \circ \KK_{\alpha_{k-d+1}}(M_g D_{\alpha_{k-d}}\rho_k D_{\alpha_{k-d}}^\dag M_g^\dag)
\\
  \rhop_{g,k}=\KK_{\alpha_{k-1}}\circ\ldots \circ \KK_{\alpha_{k-d+1}}(M_e D_{\alpha_{k-d}}\rho_k D_{\alpha_{k-d}}^\dag M_e^\dag)
\end{array}
\right.
$$
The closed-loop system, i.e.  Markov chain~\eqref{eq:main} with the causal feedback~\eqref{eq:feedmain} is still a Markov chain but with $(\rho_k,\alpha_{k-1},\cdots,\alpha_{k-d})$ as state at step $k$.  More precisely, denote by $\chi=(\rho,\beta_{1},\cdots,\beta_{d})$ this state where $\beta_l$ stands for the control $\alpha$ delayed $l$ steps. Then the state form  of the closed-loop dynamics reads
\begin{equation}\label{eq:chi}
  \left\{\begin{array}{rl}
    \rho_{k+1}&=\MM_{s_k}(\DD_{\beta_{d,k}}(\rho_k))
    \\
    \beta_{1,k+1}&= \alpha_k
    \\
    \beta_{2,k+1} &= \beta_{1,k}
    \\
    & \vdots
    \\
   \beta_{d,k+1} &= \beta_{d-1,k}.
\end{array}\right.
\end{equation}
where the control law defined by~\eqref{eq:feedmain}  corresponds to a static state feedback since
\begin{equation}
\label{eq:prediction}
\left\{
\begin{array}{l}
  \rhop_{k}=\rhop(\chi_k)=\EE{\rho_{k+d}~|~\chi_k} = \KK_{\beta_{1,k}}\circ\ldots \circ \KK_{\beta_{d,k}}(\rho_k)
\\
  \rhop_{g,k}=\rhop_g(\chi_k)=\KK_{\beta_{1,k}}\circ\ldots \circ \KK_{\beta_{d-1,k}}(M_g D_{\beta_{d,k}}\rho_k D_{\beta_{d,k}}^\dag M_g^\dag)
\\
  \rhop_{e,k}=\rhop_e(\chi_k)=\KK_{\beta_{1,k}}\circ\ldots \circ \KK_{\beta_{d-1,k}}(M_e D_{\beta_{d,k}}\rho_k D_{\beta_{d,k}}^\dag M_e^\dag)
\end{array}
\right.
\end{equation}
Notice that $\rhop_k=\rhop_{g,k}+\rhop_{e,k}$.

Simulations displayed on Figures~\ref{fig:ClosedLoop} and~\ref{fig:ClosedLoopDelay}  correspond to  100 realizations of the above
closed-loop systems with $d=0$ and $d=5$. The goal state $\bar\rho =\ket{3}\bra{3}$ contains $\bar n =3$ photons and   $\nmax$, $\varphi_0$ and ${\vartheta}$  are those used for the open-loop simulations of Figure~\ref{fig:OpenLoop}. Each  realization starts with  the same coherent state $\rho_0=\DD_{\sqrt{3}}(\ket{0}\bra{0})$ and $\beta_{1,0}=\ldots=\beta_{d,0}=0$. The feedback  parameters appearing in~\eqref{eq:feedmain} are as follows:
$$
\epsilon = \tfrac{1}{2\bar n + 1}=\tfrac{1}{7}, \quad \eta = \tfrac{1}{10}, \quad \bar\alpha=1
.
$$
This simulations illustrate the influence of the delay $d$ on the average convergence speed: the longer the delay is the slower  convergence speed becomes.
\begin{rem}\label{rem:const-cont}
The choice of the feedback law whenever $\tr{\bar\rho \rhop_k}<\eta$ might seem complicated for real-time simulation issues. However, this choice is only technical. Actually, any non-zero constant feedback law will seems to achieve the task here (see for instance the simulations of~\cite{dotsenko-et-al:PRA09}). However, the convergence proof for such simplified control scheme  is  more complicated and not considered in this paper.
\end{rem}

\subsection{Global convergence in closed-loop}
\begin{thm}\label{thm:ClosedLoop}
Take the Markov chain~\eqref{eq:chi} with the feedback~\eqref{eq:feedmain} where  $\rhop_{k}$, $\rhop_{g,k}$ and $\rhop_{e,k}$ are given by~\eqref{eq:prediction} with   $\bar \alpha >0$.  Then, for  small enough $\epsilon>0$ and $\eta>0$,  the state $\chi_k$ converges almost surely towards  $\bar \chi=(\bar \rho, 0,\ldots, 0)$  whatever the  initial condition $\chi_0 \in \XX\times \CC^d$ is  (the compact set $\XX$ is defined by~\eqref{eq:XX}).
\end{thm}
\begin{proof}
It  is based on the Lyapunov-type  function
\begin{equation}\label{eq:sub-martingale}
    V(\chi)= f\left(\tr{\bar\rho \rho^{\text{\tiny pred}}}\right) \quad\text{with}\quad \rho^{\text{\tiny pred}}=\KK_{\beta_1}\circ \ldots \circ \KK_{\beta_d} (\rho)
\end{equation}
where $f(x)=\frac{x+x^2}{2}$ has already been used during the   proof of theorem~\ref{thm:OpenLoop}.  The proof  relies in 4 lemmas:
\begin{itemize}
\item in lemma~\ref{lem:martingale},  we prove an inequality showing that,  for small enough $\epsilon$, $V(\chi)$ and $\tr{\bar\rho \rhop(\chi)}$   are sub-martingales within $\{\chi~|~ \tr{\bar\rho \rho^{\text{\tiny pred}}} \geq  \eta\}$.
\item in lemma~\ref{lem:kick}, we show that for small enough $\eta$, the trajectories starting within the set $\{\chi~|~ \tr{\bar\rho \rho^{\text{\tiny pred}}} < \eta\}$ always reach  in one step the set $\{\chi~|~ \tr{\bar\rho \rho^{\text{\tiny pred}}}\geq 2\eta\}$;
 \item in lemma~\ref{lem:doobs}, we show that  the trajectories starting within the set $\{\chi~|~ \tr{\bar\rho \rho^{\text{\tiny pred}}}\geq 2\eta\}$, will never hit the set $\{\chi~|~ \tr{\bar\rho \rho^{\text{\tiny pred}}} < \eta\}$ with a uniformly non-zero probability $p>0$;
\item in lemma~\ref{lem:invariance}, we combine  the first step and the invariance principle due to Kushner,  to prove that  almost all trajectories remaining inside $\{\chi~|~ \tr{\bar\rho \rho^{\text{\tiny pred}}} \geq  \eta\}$ converge towards $\bar \chi =(\bar\rho, 0, \ldots, 0)$.
\end{itemize}
The combination of  lemmas~\ref{lem:kick},~\ref{lem:doobs} and~\ref{lem:invariance}  shows then directly  that $\chi_k$ converges almost surely towards $\bar \chi$. We detail now these 4  lemmas.
\end{proof}
\begin{lem} \label{lem:martingale}
For $\epsilon>0$ small enough and  for  $\chi_k$ satisfying $\tr{\bar\rho \rhop(\chi_k)}\geq \eta$,
$$
  \EE{\tr{\bar\rho \rhop(\chi_{k+1})}~|~\chi_k}\geq \tr{\bar\rho \rhop(\chi_{k})}
    +\epsilon \left| \tr{\bar\rho~[\rhop_{k},\aaa]}\right|^2
$$
and also
 \begin{multline}\label{eq:martingale}
 \EE{V(\chi_{k+1})~|~\chi_k}\geq V(\chi_k)+\tfrac{\epsilon}{2} \left|\tr{\bar\rho~[\rhop_{k},\aaa]}\right|^2
\\
   +\tfrac{P_{g,k}P_{e,k}}{2}\text{\rm\huge (}
              \tr{\bar\rho~ \DD_{\alpha_k}\circ\KK_{\beta_{1,k}}\circ\cdots\circ\KK_{\beta_{d-1,k}}
              \circ\MM_g\circ\DD_{\beta_{d,k}}(\rho_k)}\\
           \qquad\qquad -
              \tr{\bar\rho~ \DD_{\alpha_k}\circ\KK_{\beta_{1,k}}\circ\cdots\circ\KK_{\beta_{d-1,k}}
              \circ\MM_e\circ\DD_{\beta_{d,k}}(\rho_k)}\text{\rm\huge )}^2
 \end{multline}
\end{lem}
\begin{proof}
Since
$M_g^\dag M_g + M_e^\dag M_e=\II$ and $[\bar\rho, M_g]=[\bar\rho,M_e]=0$, we have
\begin{multline*}
\tr{\bar\rho ~ \KK_{\beta_{1,k+1}}\circ\KK_{\beta_{2,k+1}}\circ\cdots\circ\KK_{\beta_{d,k+1}}(\rho_{k+1})}
=
\\
\tr{\bar\rho ~ \DD_{\beta_{1,k+1}}\circ\KK_{\beta_{2,k+1}}\circ\cdots\circ\KK_{\beta_{d,k+1}}(\rho_{k+1})}
.
\end{multline*}
Also, we have:
\begin{multline*}
\EE{f\left(\tr{\bar\rho ~ \KK_{\beta_{1,k+1}}\circ\KK_{\beta_{2,k+1}}\circ\cdots\circ\KK_{\beta_{d,k+1}}(\rho_{k+1})}\right)~|~\chi_k}=\\
P_{g,k}f\left(\tr{\bar\rho~\DD_{\alpha_{k}}\circ\KK_{\beta_{1,k}}\circ\cdots\circ\KK_{\beta_{d-1,k}}\circ\MM_g\circ\DD_{\beta_{d,k}}(\rho_k)}\right)+\\
P_{e,k}f\left(\tr{\bar\rho~ \DD_{\alpha_{k}}\circ\KK_{\beta_{1,k}}\circ\cdots\circ\KK_{\beta_{d-1,k}}\circ\MM_e\circ\DD_{\beta_{d,k}}(\rho_k)}\right).
\end{multline*}
By~\eqref{eq:convexe} we find
{\small
\begin{multline*}
  \EE{V(\chi_{k+1})~|~\chi_k}=
  f\left(\tr{ \bar\rho~\DD_{\alpha_{k}}\circ\KK_{\beta_{1,k}}\circ\cdots\circ\KK_{\beta_{d-1,k}}\circ\KK_{\beta_{d,k}}(\rho_k)}\right)
\\
    +\tfrac{P_{g,k}P_{e,k}}{2}\left(~\tr{\bar\rho~\DD_{\alpha_k}\circ\KK_{\beta_{1,k}}\circ\cdots\circ\KK_{\beta_{d-1,k}}
     \left(\MM_g\circ\DD_{\beta_{d,k}}(\rho_k)-\MM_e\circ\DD_{\beta_{d,k}}(\rho_k)\right)}~\right)^2
\end{multline*}}
Since $\rhop(\chi_k)=\rhop_k=\KK_{\beta_{1,k}}\circ\cdots\circ\KK_{\beta_{d-1,k}}\circ\KK_{\beta_{d,k}}(\rho_k)$
we have
{\small
$$
\tr{ \bar\rho~\DD_{\alpha_{k}}\circ\KK_{\beta_{1,k}}\circ\cdots\circ\KK_{\beta_{d-1,k}}\circ\KK_{\beta_{d,k}}(\rho_k)}=
 \tr{ \bar\rho~\DD_{\alpha_{k}}(\rhop_k)}
$$ }
For $\alpha$ small  the Baker-Campbell-Hausdorff formula yields
\begin{equation*}
\DD_\alpha(\rho)=e^{\alpha\aaa^\dag-\alpha^*\aaa}~ \rho ~e^{-(\alpha\aaa^\dag-\alpha^*\aaa)} =\rho+[\alpha\aaa^\dag-\alpha^*\aaa,\rho]+ O(|\alpha|^2)
\end{equation*}
Consequently
\begin{align*}
\tr{\bar\rho~\DD_{\alpha_k}(\rhop_k)}=
\tr{\bar\rho~\rhop_k}
+\tr{\bar\rho~ [\alpha_k\aaa^\dag-\alpha_k^*\aaa,\rhop_k]}+O(|\alpha_{k}|^2).
\end{align*}
 Since  $\alpha_k=\epsilon \tr{\bar\rho[\rhop_{k},\aaa]}$, we get
\begin{align*}
\tr{\bar\rho~\DD_{\alpha_k}(\rhop_k)}=
\tr{\bar\rho~\rhop_k} + 2\epsilon\left|\tr{\bar\rho[\rhop_{k},\aaa]}\right|^2 + O(\epsilon^2).
\end{align*}
Thus for  $\epsilon>0$ small enough and uniformly in $\rhop_k\in \XX$
$$
\tr{\bar\rho~\DD_{\alpha_k}(\rhop_k)} \geq \tr{\bar\rho~\rhop_k} + \epsilon\left|\tr{\bar\rho [\rhop_k,\aaa]}\right|^2
.
$$
Using the fact that $f$ is increasing and $f(x+y) \geq f(x)+y/2$ for any $x,y>0$, we get
$$
f\left(\tr{ \bar\rho~\DD_{\alpha_{k}}(\rhop_k)}\right) \geq
f(\left(\tr{ \bar\rho~\rhop_k}\right) +   \tfrac{\epsilon}{2} \left|\tr{\bar\rho [\rhop_k,\aaa]}\right|^2
.
$$
\end{proof}

\begin{lem}\label{lem:kick}
When   $\eta >0$  is small enough,  any  state $\chi_k$ satisfying the inequality
$\tr{\bar\rho \rhop(\chi_k)} <  \eta$ yields a new state  $\chi_{k+1}$ such that
$\tr{\bar\rho \rhop(\chi_{k+1})}\geq   2 \eta$.
\end{lem}
\begin{proof}
Since $M_g$ and $M_e$ are invertible, there exists $ \zeta\in]0,1[$ such that,  for any $\chi$,
$\tr{\rhop_g(\chi)}\geq \zeta$ and $\tr{\rhop_e(\chi)}\geq \zeta$ ($\rhop_g$ and $\rhop_e$ are defined in~\eqref{eq:prediction}).  Denote by $\XX_\zeta$ the compact  set  of Hermitian semi-definite positive matrices with trace in $[\zeta,1]$: for any $\chi$,  $\rhop_g(\chi)$ and $\rhop_e(\chi)$ are in $\XX_\zeta$. Let us prove first that, for any $\rho_g, \rho_e\in\XX_\zeta$
\begin{equation}\label{eq:maxalpha}
         \max_{|\alpha|\leq \bar\alpha}
\left( \tr{\bar\rho~ \DD_\alpha (\rho_g)}\tr{\bar\rho~ \DD_\alpha (\rho_e)} \right)>0.
\end{equation}
If for some $\rho_g,\rho_e\in\XX_\zeta$, the above maximum is zero, then for all $\alpha\in \CC$ (analyticity of $\DD_\alpha$ versus $\Re(\alpha)$ and $\Im(\alpha)$):
$$
\tr{\bar\rho~ \DD_\alpha (\rho_g)}\tr{\bar\rho~ \DD_\alpha (\rho_e)}\equiv 0
.
$$
This implies that  either $\tr{\bar\rho~ \DD_\alpha (\rho_g)}\equiv 0$ or $\tr{\bar\rho~ \DD_\alpha (\rho_e)}\equiv 0$ (if the product of two analytic functions is zero, one of them is zero). Take $\rho\in\XX_\zeta$
such that  $\tr{\bar\rho~ \DD_\alpha (\rho)}\equiv 0$. We can decompose $\rho$ as a sum of projectors,
$$\rho=\sum_{\nu=1}^m \lambda_\nu \ket{\psi_\nu}\bra{\psi_\nu},$$
where $\lambda_\nu$ are strictly positive eigenvalues, $\sum_\nu \lambda_\nu \in[\zeta,1]$,  and $\psi_\nu$ are the associated normalized eigenstates of $\rho$, $1\leq m \leq \nmax$. Since $\tr{\bar\rho~ \DD_\alpha (\rho)}\equiv 0$ for all $\alpha\in\CC$ , we have for all $\nu$,
$\bket{\psi_\nu|~D_\alpha| \bar n}=0$. Fixing one $\nu\in\{1,\cdots,m\}$ and taking $\psi=\psi_\nu$
noting that $D_\alpha=\exp(\Re(\alpha)(\aaa^\dag-\aaa)+\imath\Im(\alpha)(\aaa^\dag+\aaa)$ and deriving $j$ times versus $\Re(\alpha)$ and $\Im(\alpha)$ around $\alpha=0$ we get,
$$
\bket{\psi~|~(\aaa^\dag-\aaa)^j|\bar n}=\bket{\psi~|~(\aaa^\dag+\aaa)^j|\bar n}=0\qquad \forall j \geq 0
.
$$
With $j=0$, we get, $\bket{\psi~|\bar n}=0$. With $j=1$ we get
$\bket{\psi~|\bar n-1}=\bket{\psi~|\bar n+1}=0$ since $\aaa^\dag \ket{\bar n} = \sqrt{\bar n+1} \ket{\bar n+1}$ and
$\aaa \ket{\bar n} = \sqrt{\bar n} \ket{\bar n-1}$. With $j=2$ and using the null  Hermitian products  obtained for $j=0$ and $1$, we deduce that
$\bket{\psi~|\bar n-2}=\bket{\psi~|\bar n+2}=0$, since $\aaa \aaa^\dag \ket{\bar n}$ and $\aaa^\dag  \aaa \ket{\bar n}$  are colinear to $\ket{\bar n}$. Similarly for any $j $ and using the null  Hermitian products obtained for $j^\prime < j$, we deduce that $\bket{\psi~| \max(0,\bar n-j)}=\bket{\psi~| \min(\nmax,\bar n+j)}=0$. Thus, for any $n$, $\bket{\psi|n}=0$,  $\ket{\psi}=0$ and
we get a contradiction. Thus~\eqref{eq:maxalpha} holds true for any $\rho_g,\rho_e\in\XX_\zeta$.

The map $F$
$$(\rho_g,\rho_e) \mapsto F(\rho_g,\rho_e)=         \max_{|\alpha|\leq \bar\alpha}
\left( \tr{\bar\rho~ \DD_\alpha (\rho_g)}\tr{\bar\rho~ \DD_\alpha (\rho_e)} \right)$$
 is continuous. We have proved that for all $\rho_g,\rho_e$ in the  compact set $\XX_\zeta$, $F(\rho_g,\rho_e) >0$. Thus  exists $\delta >0$ such that $F(\rho_g,\rho_e)  \geq \delta $ for any $\rho_g,\rho_e\in\XX_\zeta$. Take  $\tilde\alpha$ an argument of the  maximum,
$$
 \tr{\bar\rho~ \DD_{\tilde\alpha}(\rho_g)}\tr{\bar\rho~ \DD_{\tilde\alpha} (\rho_e)}= \max_{|\alpha|\leq \bar\alpha}
\left( \tr{\bar\rho~ \DD_\alpha (\rho_g)}\tr{\bar\rho~ \DD_\alpha (\rho_e)} \right) \geq \delta
$$
Since (Cauchy-Schwartz inequality  for the Frobenius product)  $\tr{\bar\rho~ \DD_{\tilde\alpha}(\rho_g)}\leq 1$ and $\tr{\bar\rho~ \DD_{\tilde\alpha}(\rho_e)}\leq 1$, we have  $\tr{\bar\rho~ \DD_{\tilde\alpha}(\rho_g)} \geq\delta$ and  $\tr{\bar\rho~ \DD_{\tilde\alpha}(\rho_e)}\geq \delta$.

Take now $\eta < \tfrac{\delta}{2}$ and $\chi_k$ such that $\tr{\bar\rho \rhop(\chi_k)} \leq \eta$. According to~\eqref{eq:feedmain},
$\alpha_k$ is chosen as an  argument of
$$
\max_{|\alpha|\leq \bar\alpha}
\left( \tr{\bar\rho~ \DD_\alpha (\rhop_{g,k})}\tr{\bar\rho~ \DD_\alpha (\rhop_{e,k})} \right)
$$
where $\rhop_{g,k},\rhop_{e,k}\in\XX_\zeta$.
Thus $\tr{\bar\rho~ \DD_{\alpha_k}(\rhop_{g,k})}\geq \delta$ and $\tr{\bar\rho~ \DD_{\alpha_k}(\rhop_{e,k})} \geq \delta$. But either
$\rhop_{k+1} = \tfrac{1}{P_{g,k}}\KK_{\alpha_k}(\rhop_{g,k})$ or $ \rhop_{k+1}= \tfrac{1}{P_{e,k}} \KK_{\alpha_k}(\rhop_{g,k})$
where   $0 < P_{g,k}, P_{e,k} <1$. Since we have the identity  $\tr{\bar\rho \KK_{\alpha}(\rho)}=\tr{\bar\rho \DD_{\alpha}(\rho)}$ because $\bar\rho$ commutes with $M_g$ and $M_e$, we conclude that
$\tr{\bar\rho~ \rhop(\chi_{k+1}) }\geq \delta \geq 2 \eta$.
\end{proof}

\begin{lem}\label{lem:doobs}
Initializing the Markov process $\chi_k$  within the set $\{\chi~|~\tr{\bar\rho \rhop(\chi)}\geq 2\eta\}$, $\chi_k$ will never hit the set $\{\chi~|~\tr{\bar\rho \rhop(\chi)} < \eta\}$ with a probability
$$
p>\frac{\eta}{1-\eta}>0.
$$
\end{lem}
\begin{proof} We know from lemma~\ref{lem:martingale} that the process $1- \tr{\bar\rho \rhop(\chi)}$ is  a super martingale in the set $\{\chi~|~\tr{\bar\rho \rhop(\chi)}\geq \eta\}$. Therefore, one only needs to use the Doobs inequality recalled in appendix:
$$
\PP{\underset{0\leq k<\infty}{\sup} (1-\tr{\bar\rho \rhop(\chi_k)})> 1-\eta}> \frac{1-\tr{\bar\rho \rhop(\chi_0)}}{1-\eta}\geq \frac{1-2\eta}{1-\eta},
$$
and thus $p> 1-\frac{1-2\eta}{1-\eta}=\frac{\eta}{1-\eta}$.
\end{proof}

\begin{lem}\label{lem:invariance}
Sample paths $\chi_k$ remaining in the set $\{\tr{\bar \rho\rhop(\chi)} \geq \eta\}$ converges almost surely  to $\bar\chi$ as $k\rightarrow\infty$.
\end{lem}
\begin{proof}
We apply first  the Kushner's invariance theorem to the Markov process $\chi_k$ with the sub-martingale function $V(\chi_k)$. It ensures   convergence in probability towards  $\Inv$ the largest invariant set  attached to this sub-martingale (see appendix). Let us prove that $\Inv$ is reduced to $\{\bar\chi\}$.

 By inequality~\eqref{eq:martingale}, if $(\rho,\beta_1, \ldots,\beta_d)=\chi$ belongs to $\Inv$ then
$\tr{\bar\rho~[\rhop(\chi),\aaa]}$, i.e., $\alpha\equiv 0$ and also
\begin{multline*}
\tr{\bar\rho~ \DD_{\alpha}\circ\KK_{\beta_{1}}\circ\cdots\circ\KK_{\beta_{d-1}}
              \circ\MM_g\circ\DD_{\beta_{d}}(\rho)}=
\\
              \tr{\bar\rho~ \DD_{\alpha}\circ\KK_{\beta_{1}}\circ\cdots\circ\KK_{\beta_{d-1}}
              \circ\MM_e\circ\DD_{\beta_{d}}(\rho)}
.
\end{multline*}
Invariance  associated $\alpha\equiv 0$ implies that $\beta_1=\ldots=\beta_d=0$. Thus the above equality reads
$$
\tr{\bar\rho~ \MM_g(\rho)}=\tr{\bar\rho~ \MM_e(\rho)}
$$
where we have used the fact that, for any $\varrho\in\XX$, $\tr{\bar\rho ~\KK_0(\varrho)}= \tr{\bar\rho ~\DD_0(\varrho)}= \tr{\bar\rho \varrho}$. Then
$\rho$ satisfies
$$\tr{\bar\rho M_g\rho M_g^\dag}\tr{M_e\rho M_e^\dag}=\tr{\bar\rho M_e\rho M_e^\dag}\tr{M_g\rho M_g^\dag}$$
that reads, since $M_g^\dag\bar\rho M_g= \cos^2\!\varphi_{\bar n} ~\bar\rho$, $M_e^\dag\bar\rho M_e= \sin^2\!\varphi_{\bar n} ~\bar\rho$ and $\tr{\bar\rho \rho} >0$,
$$
\cos^2\!\varphi_{\bar n} \tr{M_e\rho M_e^\dag} = \sin^2\!\varphi_{\bar n} \tr{M_g\rho M_g^\dag}
.
$$
Since $\tr{M_e\rho M_e^\dag} +\tr{M_g\rho M_g^\dag}=1$, we recover
$
\tr{M_g\rho M_g^\dag}=\cos^2\!\varphi_{\bar n}
$
the same  condition  as the one appearing at the end of the proof of theorem~\ref{thm:OpenLoop}. Similar invariance arguments combined with $\tr{\bar\rho \rho}  >0$ imply then  $\rho=\bar\rho$. Thus $\Inv$ is reduced to $\{\bar\chi\}$.

Consider now  the event $P_{\geq \eta}=\{\forall k \geq 0, ~\tr{\bar\rho \rhop(\chi_k)}\geq \eta\}\}$. Convergence of $\chi_k$ in probability towards $\bar\chi$ means that
$$
  \forall\delta>0, \quad \lim_{k \rightarrow\infty}\PP{\|\chi_k-\bar\chi\|>\delta~|~P_{\geq \eta}}=0.
$$
where $\|\cdot\|$ is any norm on the $\chi$-space.
The continuity of $\chi\mapsto \tr{\bar\rho \rhop(\chi)}$ implies that, $\forall\delta>0$,
$$
\lim_{k \rightarrow\infty}\PP{\tr{\bar\rho \rhop(\chi_k)}<1-\delta~|~P_{\geq \eta}}=0.
$$
As $0\leq \tr{\bar\rho \rhop(\chi)}\leq 1$, we have
$$
1\geq \EE{\tr{\bar\rho \rhop(\chi_k)}~|~P_{\geq \eta}}~\geq ~ (1-\delta)\PP{1-\delta \leq \tr{\bar\rho \rhop(\chi_k)}~|~P_{\geq \eta}}
.
$$
Thus
$$
1\geq \EE{\tr{\bar\rho \rhop(\chi_k)}~|~P_{\geq \eta}}~\geq ~ 1-\delta-\PP{\tr{\bar\rho \rhop(\chi_k)}<1-\delta~|~P_{\geq \eta}}.
$$
and consequently,  $\forall\delta>0$,
$\underset{k\rightarrow\infty}{\limsup}~ \EE{\tr{\bar\rho \rhop(\chi_k)}~|~P_{\geq \eta}}\geq 1-\delta$, i.e.,
$$
\lim_{k\rightarrow\infty} \EE{\tr{\bar\rho \rhop(\chi_k)}~|~P_{\geq \eta}}=1.
$$
The process $\tr{\bar\rho\rhop(\chi_k)}$ is a bounded sub-martingale and therefore, by Theorem~\ref{thm:conv_martingale1} of the Appendix~\ref{append:stoch-stab}, we know that it converges for almost all trajectories remaining in the set $\{\tr{\bar \rho\rhop(\chi)} \geq \eta\}$. Calling the limit random variable $\text{fid}_\infty$, we have by dominated convergence theorem
$$
\EE{\text{fid}_\infty}=\EE{\lim_{k\rightarrow\infty} \tr{\bar\rho \rhop(\chi_k)}~|~P_{\geq \eta}}=\lim_{k\rightarrow \infty}\EE{\tr{\bar\rho\rhop(\chi_k)}~|~P_{\geq \eta}}=1.
$$
This trivially proves that $\text{fid}_\infty\equiv 1$ almost surely and finishes the proof of the Lemma.
\end{proof}

\subsection{Convergence rate around  the target state}\label{ssec:closed-rate}
Around the target state $\bar\chi=(\bar\rho, 0, \ldots, 0)$ the closed-loop dynamics reads
 \begin{align*}
    \rho_{k+1}&=\MM_{s_k}(\DD_{\beta_{d,k}}(\rho_k))
    \\
    \beta_{1,k+1}&= \epsilon\tr{[\aaa,\bar\rho]~\KK_{\beta_{1,k}}\circ\cdots \circ\KK_{\beta_{d,k}}(\rho_k)}
    \\
    \beta_{2,k+1} &= \beta_{1,k}
    \\
    & \vdots
    \\
   \beta_{d,k+1} &= \beta_{d-1,k}
   .
\end{align*}
Set $\chi=\bar\chi+ \delta \chi$ with $\delta\chi=(\delta \rho, \delta\beta_1, \ldots, \delta\beta_d)$ small. Computations
based on
\begin{align*}
 & \DD_{\delta\beta}(\bar\rho) =
     \bar\rho+ \left(\delta\beta [\aaa^\dag,\bar\rho]-\delta\beta^* [\aaa,\bar\rho] \right) + O(|\delta\beta|^2),
  \\
 & \KK_{\delta\beta}(\bar\rho) =
     \KK_0(\bar\rho)+ \cos{\vartheta}\left(\delta\beta [\aaa^\dag,\bar\rho]-\delta\beta^* [\aaa,\bar\rho] \right) + O(|\delta\beta|^2),
  \\
 & \KK_0(\bar\rho)=\bar\rho, \quad \KK_0 ([\aaa^\dag,\bar\rho]) = \cos{\vartheta}~ [\aaa^\dag,\bar\rho]
,\quad
  \KK_0 ([\aaa,\bar\rho]) = \cos{\vartheta}~ [\aaa,\bar\rho],
\\
& \tr{[\aaa,\bar\rho][\aaa^\dag,\bar\rho]} = -(2\bar n+1) \quad \text{and} \quad \tr{[\aaa,\bar\rho]^2}=0,
\end{align*}
yield  the following linearized closed-loop  system
{\small
\begin{equation}\label{eq:closedLoopLin}
\begin{array}{rl}
  \delta\rho_{k+1}&= A_{s_k}
      \left( \delta\rho_k+ \delta\beta_{d,k} [\aaa^\dag,\bar\rho] - \delta\beta_{d,k}^* [\aaa,\bar\rho] \right)
         A_{s_k}^\dag   -  \tr{A_{s_k}\delta\rho_k A_{s_k}^\dag}\bar\rho
\\
  \delta \beta_{1,k+1} &=
  -\epsilon (2\bar n +1) \left(\sum_{j=1}^{d} \cos^j\!{\vartheta} ~\delta\beta_{j,k}\right)
  + \epsilon \cos^d\!{\vartheta} ~\tr{\delta \rho_k [\aaa,\bar\rho]}
\\
    \delta\beta_{2,k+1} &= \delta\beta_{1,k}
    \\
    & \vdots
    \\
   \delta\beta_{d,k+1} &= \delta\beta_{d-1,k}
\end{array}
\end{equation}}
where $s_k\in\{g,e\}$, the random matrices $A_{s_k}$ are given by $A_g=\frac{M_g}{\cos\varphi_{\bar n}}$ with probability $P_g=\cos^2\!\varphi_{\bar n}$ and $A_e=\frac{M_e}{\sin\varphi_{\bar n}}$ with probability $P_e=\sin^2\!\varphi_{\bar n}$.

Set  $\delta\rho^{n_1,n_2}_k=\bket{n_1|\delta\rho_k|n_2}$ for any $n_1, n_2\in\{0, \ldots, \nmax\}$. Since $\tr{\delta\rho_k} \equiv 0$, we  exclude here the case  $(n_1,n_2)=(\bar n, \bar n)$ because
$\delta\rho_k^{\bar n, \bar n} = - \sum_{n\neq \bar n} \delta\rho_k^{n, n}$.
When  $(n_1,n_2)$ does not belong to $\{(\bar n-1,\bar n),(\bar n+1,\bar n),(\bar n,\bar n-1),(\bar n,\bar n+1)\}$,  we recover the open-loop linearized dynamics~\eqref{eq:deltarho}:
$$
\delta\rho_{k+1}^{n_1,n_2}= a_{s_k}^{n_1,n_2} \delta\rho_k^{n_1,n_2}
$$
where $s_k=g$ (resp. $s_k=e$) with probability  $\cos^2\!\varphi_{\bar n}$ (resp. $\sin^2\!\varphi_{\bar n}$) and where
$a_{g}^{n_1,n_2} =\tfrac{\cos\varphi_{n_1}\cos\varphi_{n_2}}{\cos^2\!\varphi_{\bar n}}$ and
$a_{e}^{n_1,n_2} =\tfrac{\sin\varphi_{n_1}\sin\varphi_{n_2}}{\sin^2\!\varphi_{\bar n}}$.
A direct adaptation of the proof of proposition~\ref{prop:OpenLoopLin} shows that the largest Lyapounov exponent $\Lambda_0$  of this dynamics is strictly negative and  given by
$$
\Lambda_0=\max_{ \text{\scriptsize $\begin{array}{c}
                   n\in\{0,\ldots,\nmax\} \\
                   n\neq \bar n-1, \bar n,\bar n+1
                 \end{array}$}
}
 \left(\cos^2\!\varphi_{\bar n}\log\left(\tfrac{|\cos\varphi_{n}|}{|\cos\varphi_{\bar n}|}\right)
    +\sin^2\!\varphi_{\bar n}\log\left(\tfrac{|\sin\varphi_n|}{|\sin\varphi_{\bar n}|}\right) \right)
.
$$

For $(n_1,n_2)\in \{(\bar n-1,\bar n),(\bar n+1,\bar n),(\bar n,\bar n-1),(\bar n,\bar n+1)\}$, we just have to consider $x=\delta\rho^{\bar n,\bar n-1}$ and $y=\delta\rho^{\bar n+1,\bar n}$ since $\delta\rho$ is Hermitian. Set $z_{j,k}=\delta\beta_{j,k}$. We deduce from~\eqref{eq:closedLoopLin} that the process $X_k=(x_k,y_k,z_{1,k}, \ldots,z_{d,k})$  is governed by
{\small
\begin{equation}\label{eq:xyz}
\begin{array}{rl}
  x_{k+1}&= a_{s_k} (x_k - \sqrt{\bar n} z_{d,k})
 \\
 y_{k+1}&= b_{s_k} (y_k + \sqrt{\bar n+1} z_{d,k})
\\
  z_{1,k+1} &=
  -\epsilon (2\bar n +1) \left(\sum_{j=1}^{d} \cos^j\!{\vartheta} z_{j,k}\right)
  + \epsilon \cos^d\!{\vartheta} \left(\sqrt{\bar n} x_k - \sqrt{\bar n+1} y_k \right)
  \\
    z_{2,k+1} &= z_{1,k}
    \\
    & \vdots
    \\
   z_{d,k+1} &= z_{d-1,k}
\end{array}
\end{equation}
}
where $s_k=g$ (resp. $s_k=e$) with probability  $\cos^2\!\varphi_{\bar n}$ (resp. $\sin^2\!\varphi_{\bar n}$) and
$$
  a_{g} =\tfrac{\cos\varphi_{\bar n-1}}{\cos\varphi_{\bar n}},\quad
  a_{e} =\tfrac{\sin\varphi_{\bar n-1}}{\sin\varphi_{\bar n}},\quad
  b_{g} =\tfrac{\cos\varphi_{\bar n+1}}{\cos\varphi_{\bar n}},\quad
  b_{e} =\tfrac{\sin\varphi_{\bar n+1}}{\sin\varphi_{\bar n}}
.
$$
Take $\mu >0$ to be defined later, set $\sigma=|\cos{\vartheta}| \in]0,1[$ and consider
 $$
V(X)=|x|+|y|+\mu\left(|z_{1}|+\sigma|z_{2}|+\cdots+\sigma^{d-1}|z_{d}|\right). $$
A direct computation exploiting~\eqref{eq:xyz} yields
{\small
\begin{multline*}
\EE{V(X_{k+1})~|~X_k} = \sigma |x_k - \sqrt{\bar n} z_{d,k}| + \sigma |y_k + \sqrt{\bar n+1} z_{d,k} |
\\
+\sigma \mu \left(|z_{1,k}|+\sigma|z_{2,k}|+\cdots+\sigma^{d-2}|z_{d-1,k}|\right)
\\
+ \epsilon \mu \left|  - (2\bar n +1) \left(\sum_{j=1}^{d} \cos^j\!{\vartheta} z_{j,k}\right)
  +  \cos^d\!{\vartheta} \left(\sqrt{\bar n} x_k - \sqrt{\bar n+1} y_k \right)
 \right|
.
\end{multline*}
}
Thus
{\small
\begin{multline*}
\EE{V(X_{k+1})~|~X_k} \leq (\sigma+ \epsilon\mu \sigma^{d}\sqrt{\bar n+1} )(|x_k | +  |y_k  |)
\\
+\sigma(1+\epsilon(2\bar n +1)) \mu
 \left(|z_{1,k}|+\cdots+\sigma^{d-2}|z_{d-1,k}|
    + \tfrac{\sigma^{1-d}(\sqrt{\bar n}+\sqrt{\bar n+1}) + \epsilon\mu(2\bar n +1)}{(1+\epsilon(2\bar n +1)) \mu }
       ~\sigma^{d-1}|z_{d,k}|\right)
.
\end{multline*}
}
Take $\mu=\tfrac{\sqrt{\bar n}+\sqrt{\bar n+1}}{\sigma^{d-1}}$, then
$$
\EE{V(X_{k+1})~|~X_k} \leq \sigma(1+2\epsilon(\bar n +1))~ V(X_k).
$$
Because  $\sigma < 1$, for $\epsilon>0$ small enough ($\epsilon < \tfrac{1-\sigma}{2(\bar n +1)}$) the norm  $V(X_k)$ is a  super-martingale converging exponentially  almost surely towards zero. Thus the largest Lyapunov exponent of the linear Markov chain~\eqref{eq:xyz} is strictly negative.
To conclude,  we have proved the following proposition:
\begin{prop}\label{prop:ClosedLoopLin}
Consider  the linear Markov chain~\eqref{eq:closedLoopLin}. For  small enough $\epsilon>0$, its largest Lyapunov exponent is strictly negative.
\end{prop}

\section{Quantum filter and separation principle}\label{sec:filter}

\subsection{Quantum filter and closed-loop simulations}\label{ssec:filter}
The feedback law~\eqref{eq:feedmain} requires  the knowledge of $(\rho_k,\beta_{1,k},\cdots,\beta_{d,k})$. When the measurement process is fully efficient and the jump model~\eqref{eq:main} admits no error, the Markov system~\eqref{eq:chi} represents a natural choice for the quantum filter to estimate the value of $\rho$. Indeed, we define the estimator $\chie_k=  (\rhoe_k,\beta_{1,k},\cdots,\beta_{d,k})$ satisfying the dynamics
\begin{equation}\label{eq:filter}
  \left\{\begin{array}{rl}
    \rhoe_{k+1}&=\MM_{s_k}(\DD_{\beta_{d,k}}(\rhoe_k))
    \\
    \beta_{1,k+1}&= \alpha_k
    \\
    \beta_{2,k+1} &= \beta_{1,k}
    \\
    & \vdots
    \\
   \beta_{d,k+1} &= \beta_{d-1,k}.
\end{array}\right.
\end{equation}
Note that, similarly to any observer-controller structure, the jump result, $s_k=g$ or $e$, is the output of the physical system~\eqref{eq:main} but the feedback control $\alpha_k$ is a function of the estimator $\rhoe$. Indeed, $\alpha_k$ is defined as in~\eqref{eq:feedmain}:\small
\begin{equation}\label{eq:feedfilter}
    \alpha_{k}=\left\{
    \begin{array}{ll}
        \epsilon \tr{\bar\rho~ [\rhope_{k},\aaa]}
         & \mbox{ if } \tr{\bar\rho\rhope_{k}} \ge \eta
         \\
     \underset{|\alpha|\le\bar\alpha}{\text{argmax}}
        \left( \tr{\bar\rho~\DD_\alpha(\rhope_{g,k})}\tr{\bar\rho~\DD_\alpha(\rhope_{e,k})}\right)
          & \mbox{ if } \tr{\bar\rho\rhope_{k}} < \eta  \\
    \end{array}
    \right.
\end{equation}\normalsize
where the predictor's state $\rhope_k$ is defined as follows:
$$
\left\{
\begin{array}{l}
 \rhope_k=\KK_{\alpha_{k-1}}\circ\ldots \circ \KK_{\alpha_{k-d}}(\rhoe_k)
\\
  \rhope_{g,k}=\KK_{\alpha_{k-1}}\circ\ldots \circ \KK_{\alpha_{k-d+1}}(M_g D_{\alpha_{k-d}}\rhoe_k D_{\alpha_{k-d}}^\dag M_g^\dag)
\\
  \rhope_{g,k}=\KK_{\alpha_{k-1}}\circ\ldots \circ \KK_{\alpha_{k-d+1}}(M_e D_{\alpha_{k-d}}\rhoe_k D_{\alpha_{k-d}}^\dag M_e^\dag)
\end{array}
\right.
$$
We will see through this section that, even if do not have any a priori knowledge of the initial state of the physical system, the choice of the feedback law through the above quantum filter can ensure the convergence of the system towards the desired Fock state. Indeed, we prove a semi-global robustness of the feedback scheme with respect to the choice of the initial state of the quantum filter.

Before going through the details of  this robustness analysis, let us illustrate it through some numerical simulations. In the simulations of Figure~\ref{fig:filter1}, we assume no a priori knowledge on the initial state of the system. Therefore, we initialize the filter equation at the maximally mixed state $\rhoe_0=\frac{1}{\nmax+1}\II_{(\nmax+1)\times (\nmax+1)}$. Computing the feedback control through the above quantum filter and injecting it to the physical system modeled by~\eqref{eq:main}, the fidelity (with respect to the target Fock state) of the closed-loop trajectories of the physical system are illustrated in the first plot of Figure~\ref{fig:filter1}. The second plot of this figure, illustrate the Frobenius distance between the estimator $\rhoe$ and the physical state $\rho$. As one can easily see, one still have the convergence of the quantum filter and the physical system to the desired Fock state (here $\ket{3}\bra{3}$).

Through these simulations, we have considered the same measurement and control parameters as those of Section~\ref{sec:closed-loop}. The system is initialized at the coherent state $\rho_0=\DD_{\sqrt{3}}(\ket{0} \bra{0})$ while the quantum filter is initialized at $\rhoe_0=\frac{1}{\nmax+1}\II_{(\nmax+1)\times (\nmax+1)}$.

\begin{figure}
  % Requires \usepackage{graphicx}
 \centerline{\includegraphics[width=.6\textwidth]{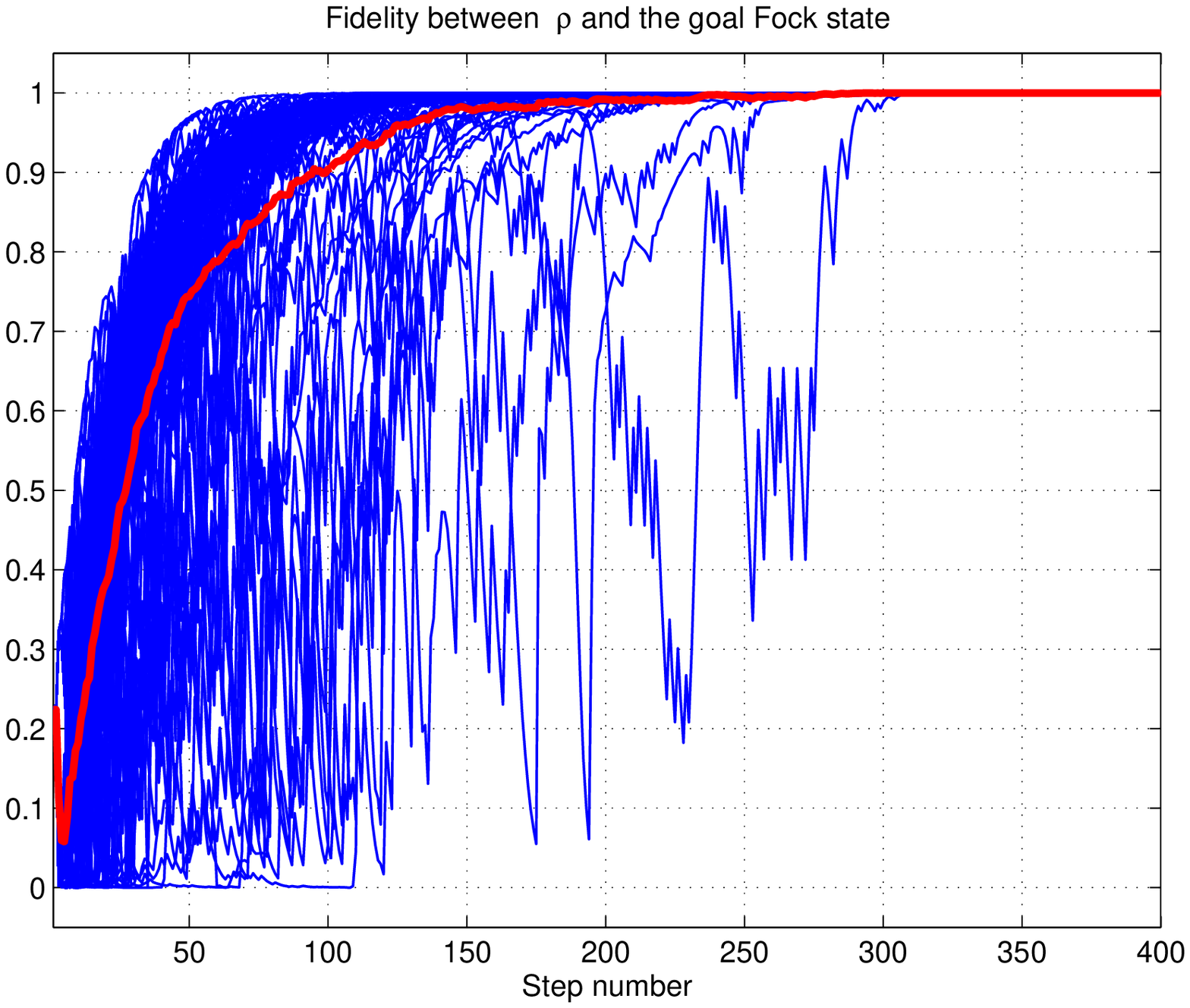}}
 \centerline{\includegraphics[width=.6\textwidth]{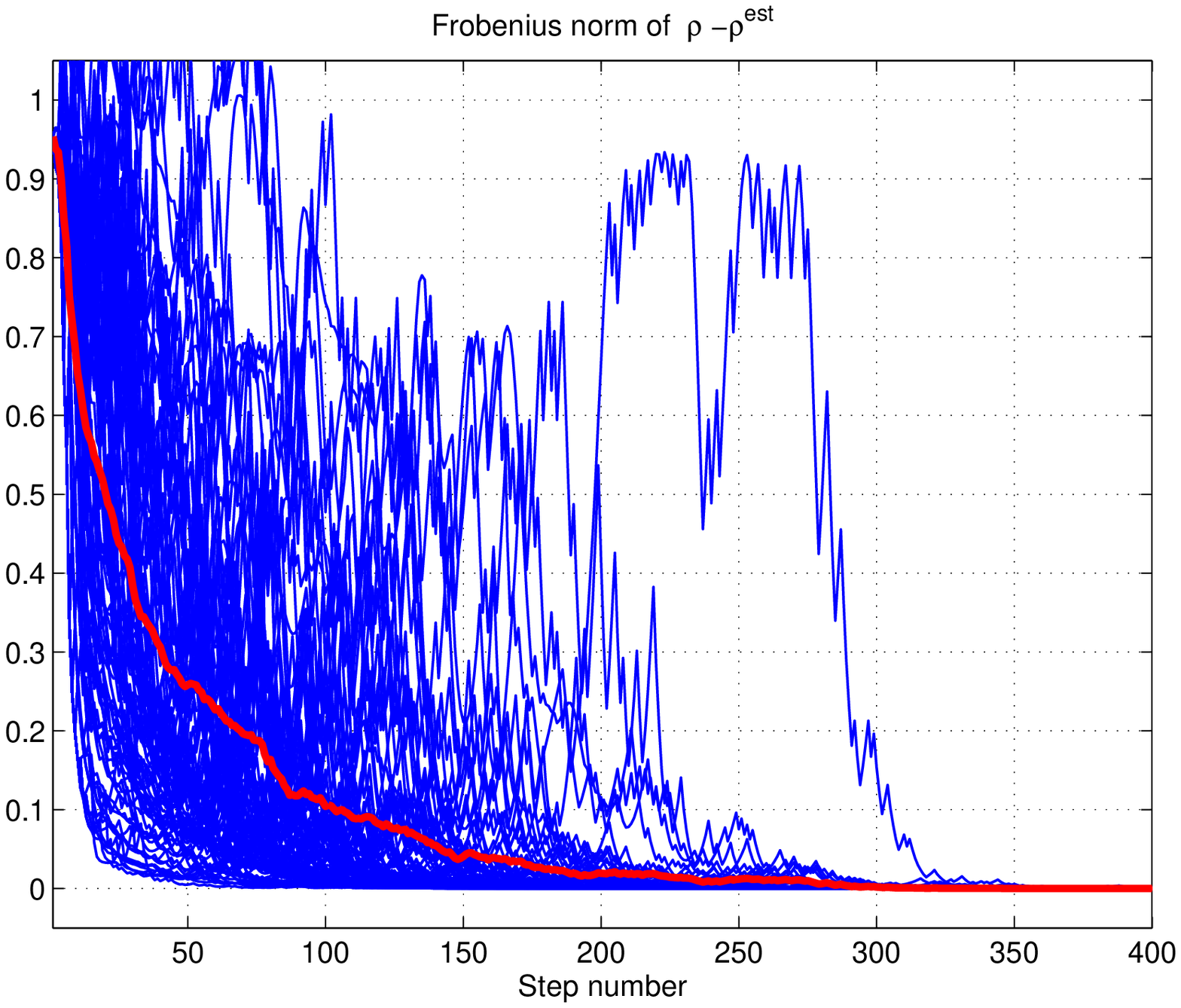}}
 \caption{ (First plot) $\tr{\rho_k\bar\rho}=\bket{3|\rho_k|3}$  versus $k\in\{0, \ldots, 400\}$ for 100 realizations of  the closed-loop Markov process~\eqref{eq:main} with feedback~\eqref{eq:feedfilter} based on the quantum filter~\eqref{eq:filter} starting from the same  state $\chie_0=(\frac{1}{\nmax+1}\II_{(\nmax+1)\times (\nmax+1)},0,\ldots,0)$ with $5$-step delay ($d=5$). The initial state of the physical system $\rho_0$ is given by  $\DD_{\sqrt{3}}(\ket{0} \bra{0})$. The ensemble average over these realizations corresponds to the thick red curve; (Second plot) The Frobenius distance between the estimator $\rhoe$ and $\rho$ ($\sqrt{\tr{(\rho-\rhoe)^2}}$) for 100 realizations. The ensemble average over these realizations corresponds to the thick red curve.}
\label{fig:filter1}
\end{figure}

Through the next subsection, we establish a sort of separation principle implying this semi-global robustness of the closed-loop system with respect to the initial state of the filter equation. Also through the short Subsection~\ref{ssec:filter-rate} we provide a heuristic analysis of the local convergence rate of the filter equation around the target Fock state.

\subsection{A quantum separation principle}\label{ssec:separation}
We consider the joint system-observer dynamics defined for the state $\Xi_k= (\rho_k,\rhoe_k,\beta_{1,k},\ldots,\beta_{d,k})$:
\begin{equation}\label{eq:joint}
  \left\{\begin{array}{rl}
    \rho_{k+1}&=\MM_{s_k}(\DD_{\beta_{d,k}}(\rho_k))
    \\
    \rhoe_{k+1}&=\MM_{s_k}(\DD_{\beta_{d,k}}(\rhoe_k))
    \\
    \beta_{1,k+1}&= \alpha_k
    \\
    \beta_{2,k+1} &= \beta_{1,k}
    \\
    & \vdots
    \\
   \beta_{d,k+1} &= \beta_{d-1,k}.
\end{array}\right.
\end{equation}
We have the following result, a quantum version  of the separation principle  ensuring asymptotic stability of observer/controller from stability of  the observer and of the controller separatly.
\begin{thm}\label{thm:separation}
Consider any closed-loop system of the form~\eqref{eq:joint}, where the feedback law $\alpha_k$ is a function of the quantum filter: $\alpha_k=g(\rhoe_k,\beta_{1,k},\ldots,\beta_{d,k})$. Assume moreover that, whenever $\rhoe_0=\rho_0$ (so that the quantum filter coincides with the closed-loop dynamics~\eqref{eq:chi}), the closed-loop system converges almost surely towards a fixed pure state $\bar \rho$. Then, for any choice of the initial state $\rhoe_0$, such that   $\text{ker}\rhoe_0 \subset \text{ker}\rho_0$, the trajectories of the system  converge almost surely towards the same pure state: $\rho_k\rightarrow \bar\rho$.
\end{thm}
\begin{rem}\label{rem:ker}
One only needs to choose $\rhoe_0=\frac{1}{\nmax+1}\II_{(\nmax+1)\times (\nmax+1)}$, so that the assumption $\text{ker}\rhoe_0 \subset \text{ker}\rho_0$ is satisfied for any $\rho_0$.
\end{rem}

\begin{proof}
The basic idea  is based on the fact that $\EE{\tr{\rho_k \bar\rho}~|~\rho_0,\rhoe_0}$ (where we take the expectation over all jump realizations) depends linearly on $\rho_0$ even though we are applying a feedback control. Indeed, the feedback law $\alpha_k$ depends only on the historic of the quantum jumps as well as the initialization of the quantum filter $\rhoe_0$. Therefore, we can write
$$
\beta_{k,d}=\alpha_{k-d}=\alpha(\rhoe_0,s_0,\ldots,s_{k-d-1}),
$$
where $\{s_j\}_{j=0}^{k-1}$ denotes the sequence of $k$ first jumps. Finally, through simple computations, we have
$$
\EE{\tr{\rho_k \bar\rho}~|~\rho_0,\rhoe_0}= \sum_{s_0,\ldots,s_{k-1}}  \widetilde \MM_{s_{k-1}}\circ \DD_{\beta_{k,d}}\circ \ldots\circ  \widetilde \MM_{s_0}\circ \DD_{\beta_{0,d}}\rho_0,
$$
where
$$
\widetilde \MM_{s} \rho= M_s\rho M_s^\dag.
$$
So, we easily have the linearity of $\EE{\tr{\rho_k \bar\rho}~|~\rho_0,\rhoe_0}$ with respect to $\rho_0$.

At this point, we apply the assumption $\text{ker}\rhoe_0 \subset \text{ker}\rho_0$ and therefore, one can find a constant $\gamma>0$ and a well-defined density matrix $\rho_0^c$ in $\XX$, such that
$$
\rhoe_0=\gamma\rho_0+(1-\gamma)\rho_0^c.
$$
Now, considering the system~\eqref{eq:joint} initialized at the state $(\rhoe_0,\rhoe_0,0,\ldots,0)$, we have by the assumptions of the theorem and applying  dominated convergence theorem:
$$
\lim_{k\rightarrow\infty} \EE{\tr{\rho_k \bar\rho}~|~\rhoe_0,\rhoe_0}=1.
$$
By the linearity of $ \EE{\tr{\rho_k \bar\rho}~|~\rho_0,\rhoe_0}$  with respect to $\rho_0$, we have
$$
\EE{\tr{\rho_k \bar\rho}~|~\rhoe_0,\rhoe_0}=\gamma \EE{\tr{\rho_k \bar\rho}~|~\rho_0,\rhoe_0}+(1-\gamma)\EE{\tr{\rho_k \bar\rho}~|~\rho_0^c,\rhoe_0},
$$
and as both $ \EE{\tr{\rho_k \bar\rho}~|~\rho_0,\rhoe_0}$ and $ \EE{\tr{\rho_k \bar\rho}~|~\rho^c_0,\rhoe_0}$ are less than or equal to one, we necessarily have that both of them converge to 1:
$$
\lim_{k\rightarrow\infty} \EE{\tr{\rho_k \bar\rho}~|~\rho_0,\rhoe_0}=1.
$$
This implies the almost sure convergence of the physical system towards the pure state $\bar\rho$.
\end{proof}

\subsection{Local convergence rate for the quantum filter}\label{ssec:filter-rate}
Let us  linearize the system-observer dynamics~\eqref{eq:joint} around the equilibrium state \\$\bar\Xi=(\bar\rho,\bar\rho,0,\ldots,0)$.  Set $\Xi=\bar\Xi+\delta\Xi$ with
$\delta\Xi=(\delta\rho, \delta\rhoe, \delta\beta_1, \ldots,\delta \beta_d )$ small, $\delta \rho$ and $\delta\rhoe$ Hermitian and of trace 0. We have the following dynamics for the linearized system (adaptation of~\eqref{eq:closedLoopLin}):
{\small
\begin{equation}\label{eq:filterLin}
\begin{array}{rl}
  \delta\rho_{k+1}&= A_{s_k}
      \left( \delta\rho_k+ \delta\beta_{d,k} [\aaa^\dag,\bar\rho] - \delta\beta_{d,k}^* [\aaa,\bar\rho] \right)
         A_{s_k}^\dag   -  \tr{A_{s_k}\delta\rho_k A_{s_k}^\dag}\bar\rho
\\
  \delta\rhoe_{k+1}&= A_{s_k}
      \left( \delta\rhoe_k+ \delta\beta_{d,k} [\aaa^\dag,\bar\rho] - \delta\beta_{d,k}^* [\aaa,\bar\rho] \right)
         A_{s_k}^\dag   -  \tr{A_{s_k}\delta\rhoe_k A_{s_k}^\dag}\bar\rho
\\
  \delta \beta_{1,k+1} &=
  -\epsilon (2\bar n +1) \left(\sum_{j=1}^{d} \cos^j\!{\vartheta} ~\delta\beta_{j,k}\right)
  + \epsilon \cos^d\!{\vartheta} ~\tr{\delta \rhoe_k [\aaa,\bar\rho]}
\\
    \delta\beta_{2,k+1} &= \delta\beta_{1,k}
    \\
    & \vdots
    \\
   \delta\beta_{d,k+1} &= \delta\beta_{d-1,k}
\end{array}
\end{equation}}
where $s_k\in\{g,e\}$, the random matrices $A_{s_k}$ are given by $A_g=\frac{M_g}{\cos\varphi_{\bar n}}$ with probability $P_g=\cos^2\!\varphi_{\bar n}$ and $A_e=\frac{M_e}{\sin\varphi_{\bar n}}$ with probability $P_e=\sin^2\!\varphi_{\bar n}$.

At this point, we note that by considering $\widetilde{\delta\rho}_k= \delta \rhoe_k-\delta\rho_k$,  we have the following simple dynamics:
$$
\widetilde{\delta\rho}_{k+1}=A_{s_k}\widetilde{\delta\rho}_k A_{s_k}^\dag -\tr{A_{s_k}\widetilde{\delta\rho}_k A_{s_k}^\dag}\bar\rho.
$$
Indeed, as the same control laws are applied to the quantum filter and the physical system, the difference between $\delta \rhoe_k$ and $\delta \rho_k$ follows the same dynamics as the linearized  open-loop system~\eqref{eq:OpenLoopLin}. But, we know by the proposition~\ref{prop:OpenLoopLin} that this linear system admits  strictly negative Lyapunov exponents. This triangular structure,  together with the convergence rate analysis of the closed-loop system in proposition~\ref{prop:ClosedLoopLin},   yields the following proposition whose detailed proof is left to the reader:
\begin{prop} \label{prop:filterlin}
Consider  the linear Markov chain~\eqref{eq:filterLin}. For  small enough $\epsilon>0$, its largest Lyapunov exponent is strictly negative.
\end{prop}

\section{Conclusion}\label{sec:conclusion}
We have analyzed  a measurement-based feedback control allowing to stabilize globally and deterministically a desired Fock state. In this feedback design, we have taken into account the important delay between the measurement process and the feedback injection. This delay has been compensated by a stochastic version of a Smith predictor in the quantum filtering equation.

In fact, the measurement process of the experimental setup~\cite{deleglise-et-al:nature08} admits some other imperfections. These imperfections can, essentially, be resumed to the following ones: 1- the atom-detector is not fully efficient and it can miss some of the atoms (about 20\%); 2- the atom-detector is not fault-free and the result of the measurement (atom in the state $g$ or $e$) can be inter-changed (a fault rate of about 10\%); 3- the atom preparation process is itself a stochastic process following a Poisson law and therefore the measurement pulses can be empty of atom (a pulse occupation rate of about 40\%). The knowledge of all these rates can help us to adapt the quantum filter by taking into account these imperfections. This has been done in~\cite{dotsenko-et-al:PRA09}, by considering the Bayesian law and providing   numerical evidence of the efficiency of such feedback algorithms assuming all these imperfections.

\bibliographystyle{plain}
% \bibliography{E:/Latex/rhn}

\begin{thebibliography}{10}

\bibitem{belavkin-92}
V.P. Belavkin.
\newblock Quantum stochastic calculus and quantum nonlinear filtering.
\newblock {\em Journal of Multivariate Analysis}, 42(2):171--201, 1992.

\bibitem{braginsky-vorontosov:75}
V.B. Braginskii and Y.I. Vorontsov.
\newblock Quantum-mechanical limitations in macroscopic experiments and modern
  experimental technique.
\newblock {\em Sov. Phys. Usp.}, 17(5):644--650, 1975.

\bibitem{brune-et-al:PhRevA92}
M.~Brune, S.~Haroche, J.-M. Raimond, L.~Davidovich, and N.~Zagury.
\newblock Manipulation of photons in a cavity by dispersive atom-field
  coupling: Quantum-nondemolition measurements and g{\'e}n{\'e}ration of
  "{S}chr{\"o}dinger cat" states.
\newblock {\em Physical Review A}, 45(7):5193--5214, 1992.

\bibitem{deleglise-et-al:nature08}
S.~Del{\'e}glise, I.~Dotsenko, C.~Sayrin, J.~Bernu, M.~Brune, J.-M. Raimond,
  and S.~Haroche.
\newblock Reconstruction of non-classical cavity field states with snapshots of
  their decoherence.
\newblock {\em Nature}, 455:510--514, 2008.

\bibitem{doherty-jacobs-99}
A.C. Doherty and K.~Jacobs.
\newblock Feedback control of quantum systems using continuous state
  estimation.
\newblock {\em Phys. Rev. A}, 6:2700--2711, 1999.

\bibitem{dotsenko-et-al:PRA09}
I.~Dotsenko, M.~Mirrahimi, M.~Brune, S.~Haroche, J.-M. Raimond, and P.~Rouchon.
\newblock Quantum feedback by discrete quantum non-demolition measurements:
  towards on-demand generation of photon-number states.
\newblock {\em Physical Review A}, 80: 013805-013813, 2009.

\bibitem{geremia:PRL06}
J.M. Geremia.
\newblock Deterministic and nondestructively verifiable preparation of photon
  number states.
\newblock {\em Physical Review Letters}, 97(073601), 2006.

\bibitem{gleyzes-et-al:nature07}
S.~Gleyzes, S.~Kuhr, C.~Guerlin, J.~Bernu, S.~Del{\'e}glise, U.~Busk Hoff,
  M.~Brune, J.-M. Raimond, and S.~Haroche.
\newblock Quantum jumps of light recording the birth and death of a photon in a
  cavity.
\newblock {\em Nature}, 446:297--300, 2007.

\bibitem{guerlin-et-al:nature07}
C.~Guerlin, J.~Bernu, S.~Del{\'e}glise, C.~Sayrin, S.~Gleyzes, S.~Kuhr,
  M.~Brune, J.-M. Raimond, and S.~Haroche.
\newblock Progressive field-state collapse and quantum non-demolition photon
  counting.
\newblock {\em Nature}, 448:889--893, 2007.

\bibitem{vanhandel:ieee05}
R.~Van Handel, J.~K. Stockton, and H.~Mabuchi.
\newblock Feedback control of quantum state reduction.
\newblock {\em IEEE Trans. Automat. Control}, 50:768--780, 2005.

\bibitem{haroche-raimond:book06}
S.~Haroche and J.M. Raimond.
\newblock {\em Exploring the Quantum: Atoms, Cavities and Photons.}
\newblock Oxford University Press, 2006.

\bibitem{kailath-book}
T.~Kailath.
\newblock {\em Linear Systems}.
\newblock Prentice-Hall, Englewood Cliffs, NJ, 1980.

\bibitem{kushner-71}
H.J. Kushner.
\newblock {\em Introduction to stochastic control}.
\newblock Holt, Rinehart and Wilson, INC., 1971.

\bibitem{Liptser-Shiryayev}
R.S. Liptser and A.N. Shiryayev.
\newblock {\em Statistics of Random Processes I General Theory}.
\newblock Springer-Verlag, 1977.

\bibitem{mirrahimi-et-al:cdc09}
M.~Mirrahimi, I.~Dotsenko, and P.~Rouchon.
\newblock Feedback generation of quantum fock states by discrete qnd measures.
\newblock In {\em Decision and Control, 2009 held jointly with the 2009 28th
  Chinese Control Conference. CDC/CCC 2009. Proceedings of the 48th IEEE
  Conference on}, pages 1451 --1456, 2009.

\bibitem{mirrahimi-handel:siam07}
M.~Mirrahimi and R.~Van Handel.
\newblock Stabilizing feedback controls for quantum systems.
\newblock {\em SIAM Journal on Control and Optimization}, 46(2):445--467, 2007.

\bibitem{smith-58}
O.J.M. Smith.
\newblock Closer control of loops with dead time.
\newblock {\em Chemical Engineering Progress}, 53(5):217--219, 1958.

\bibitem{thorne-et-al:78}
K.S. Thorne, R.W.P. Drever, C.M. Caves, M.~Zimmermann, and V.D. Sandberg.
\newblock Quantum nondemolition measurements of harmonic oscillators.
\newblock {\em Phys. Rev. Lett.}, 40:667--671, 1978.

\bibitem{vitali-95}
P.~Tombesi and D.~Vitali.
\newblock Macroscopic coherence via quantum feedback.
\newblock {\em Phys. Rev. A}, 51:4913--4917, 1995.

\bibitem{Unruh-78}
W.G. Unruh.
\newblock Analysis of quantum-nondemolition measurement.
\newblock {\em Phys. Rev. D}, 18:1764--1772, 1978.

\bibitem{wiseman-94}
H.M. Wiseman.
\newblock Quantum theory of continuous feedback.
\newblock {\em Physical Review A}, 49:2133--50, 1994.

\end{thebibliography}

\appendix

\section{Stability theory for stochastic processes}\label{append:stoch-stab}

We recall here the Doob's first martingale convergence theorem, the Doob's inequality and  the Kushner's invariance theorem. For detailed discussions and proofs we refer to~\cite{Liptser-Shiryayev}(Chapter 2) and ~\cite{kushner-71} (Sections 8.4 and 8.5).

The following theorem characterizes the convergence of bounded martingales:
\begin{thm}[Doob's first martingale convergence theorem]\label{thm:conv_martingale1}
Let $\{X_n\}$ be a Markov chain on state space $\XX$ and suppose that
$$
\EE{X_n}\geq \EE{X_m},\qquad\text{for }n\geq m,
$$
this is $X_n$ is a submartingale. Assume furthermore that ($x^+$ is the positive part of $x$)
$$
\sup_n \EE{X_n^+}<\infty.
$$
Then $\lim_n X_n$ ($=X_\infty$) exists with probability $1$, and $\EE{X_\infty^+} <\infty$.
\end{thm}

Now, we recall two results that are often referred as the stochastic versions of the Lyapunov stability theory and the LaSalle's invariance principle.

\begin{thm}[Doob's Inequality]\label{thm:doob}
Let $\{X_n\}$ be a Markov chain on state space $\XX$. Suppose that there is a non-negative function $V(x)$ satisfying
$
\EE{V(X_1)~|~X_0=x}-V(x)=-k(x),
$
where $k(x)\geq 0$ on the set $\{s:V(x)<\lambda\}\equiv Q_\lambda$. Then
$$
\PP{\underset{\infty>n\geq 0}{\sup} V(X_n)\geq \lambda~|~X_0=x}\leq \frac{V(x)}{\lambda}.
$$
%Furthermore, there is some random $v\geq 0$, so that for paths never leaving $Q_\lambda$, $V(X_n)\rightarrow v\geq 0$ almost surely.
\end{thm}
For the statement of the second theorem, we need to use the language of probability measures rather than the random processes.
Therefore, we deal with the space $\mathcal M$ of probability measures on the state space $\XX$. Let $\mu_0=\sigma$ be the initial probability distribution (everywhere through this paper we have dealt with the case where $\mu_0$ is a Dirac on a state $\rho_0$ of the state space of density matrices). Then, the probability distribution of $X_n$, given initial distribution $\sigma$, is to be denoted by $\mu_n(\sigma)$. Note that for $m\geq 0$, the Markov property implies:
$
\mu_{n+m}(\sigma)=\mu_n(\mu_m(\sigma)).
$
\begin{thm}[Kushner's invariance theorem]\label{thm:kushner}
Consider the same assumptions as that of the theorem~\ref{thm:doob}. Let $\mu_0=\sigma$ be concentrated on a state $x_0\in Q_\lambda$  ($Q_\lambda$ being defined as in theorem~\ref{thm:doob}), i.e. $\sigma(x_0)=1$.  Assume that $0\leq  k(X_n)\rightarrow 0$ in $Q_\lambda$ implies that $X_n\rightarrow \{x~|~k(x)=0\}\cap Q_\lambda\equiv K_\lambda$. Under the conditions of theorem~\ref{thm:doob}, for trajectories never leaving $Q_\lambda$, $X_n$ converges to $K_\lambda$ almost surely. Also, the associated conditioned probability measures $\tilde\mu_n$ tend to the largest invariant set of measures ${\mathcal M}_\infty\subset {\mathcal M}$ whose support set is in $K_\lambda$. Finally, for the trajectories never leaving $Q_\lambda$, $X_n$ converges, in probability, to the support set of ${\mathcal M}_\infty$.
\end{thm}

\section{Lyapunov exponents of linear stochastic processes}\label{append:Lyap-exp}

Consider a discrete-time linear stochastic system defined on $\RR^d$ by
$$
X_{k+1}=A_{s_k} X_k,
$$
where $A_{s_k}$ is a random matrix taking its values inside a finite set $\{A_1,\ldots,A_m\}$ with a stationary probability distribution for $s_k$ over $\{1,\ldots,m\}$. Then
$$
\lambda(X_0)=\lim_{k\rightarrow\infty}\frac{1}{k}\log\left( \frac{\|X_k\|}{\|X_0\|}\right),
$$
for different initial states $X_0\in \RR^d$, may take at most $d$ values which are called the Lyapunov exponents of the linear stochastic system.

\end{document}